\documentclass[a4paper,10pt]{article}
\usepackage[english]{babel}
\usepackage{a4wide}
\usepackage{enumitem}
\usepackage{multicol}
\usepackage[utf8]{inputenc}
\usepackage{todonotes}

\usepackage{amsmath,amssymb,amsthm,array,textpos,float,dsfont,verbatim,fancyvrb,adjustbox,xfp,parskip}
\usepackage{breqn}

\usepackage{mathtools,makecell,graphicx,wrapfig,caption,tikz,tikz-cd,framed,pgffor}
\usetikzlibrary{arrows,positioning,decorations.pathmorphing,decorations.markings,arrows.meta, angles, calc, quotes, tikzmark,chains,shapes.geometric,bending}
\usepackage[labelformat=simple]{subcaption}

\tikzset{
  symbol/.style={
    draw=none,
    every to/.append style={
      edge node={node [sloped, allow upside down, auto=false]{$#1$}}}
  }
}
\tikzset{edge/.style= {line width=0.75pt} }
\tikzset{commutative diagrams/.cd,
	mysymbol/.style={start anchor=center,end anchor=center,draw=none}
}
\PassOptionsToPackage{hyphens}{url}\usepackage{hyperref}
\usepackage{cleveref}

\newtheoremstyle{break}
{\topsep}{\topsep}%
{\itshape}{}%
{\bfseries}{}%
{\newline}{}%

\theoremstyle{definition}
\newtheorem{definition}{Definition}[section]
\newtheorem{example}[definition]{Example}
\newtheorem{question}{Question}

\theoremstyle{plain}
\newtheorem{theorem}[definition]{Theorem}
\newtheorem{proposition}[definition]{Proposition}
\newtheorem{corollary}[definition]{Corollary}
\newtheorem{lemma}[definition]{Lemma}

\theoremstyle{remark}
\newtheorem{remark}[definition]{Remark}

\newtheorem{notation}[definition]{Notation}

\numberwithin{equation}{section}

\DeclareMathOperator{\gr}{gr}
\DeclareMathAlphabet\wiskunde{U}{msb}{m}{n}
\newcommand{\N}{{\wiskunde N}}
\newcommand{\Z}{{\wiskunde Z}}
\newcommand{\Q}{{\wiskunde Q}}

\newcommand{\C}{{\wiskunde C}}

\DeclareMathAlphabet\gotisch{U}{euf}{m}{n}

\newcommand{\n}{\mathfrak{n}}
\newcommand{\Aut}{{\rm Aut}}

\newcommand{\Gl}{{\rm GL}}
\newcommand{\Spec}{{\rm Spec}_R}
\newcommand{\Sp}{{\rm Sp}}
\newcommand{\Id}{{\rm Id}}
\renewcommand{\Im}{{\rm Im}}
\newcommand*{\olprec}{\mathbin{\overline\prec}}
\newcommand*{\olGamma}{\overline\Gamma}
\newcommand*{\olE}{\overline{E}}

\newcommand*{\olvarphi}{\overline\varphi}

\makeatletter
\newcommand{\biggg}{\bBigg@\thr@@}
\newcommand{\Biggg}{\bBigg@{3}}

\makeatother

\makeatletter
\renewcommand*\env@matrix[1][*\c@MaxMatrixCols c]{%
  \hskip -\arraycolsep
  \let\@ifnextchar\new@ifnextchar
  \array{#1}}
\makeatother

\title{$R_{\infty}$--property for groups commensurable to nilpotent quotients of RAAGs}
\author{Maarten Lathouwers\thanks{KU Leuven Kulak Kortrijk Campus, Department of Mathematics. Email: maarten.lathouwers@kuleuven.be. The author was supported by a PhD fellowship of the Research Fund - Flanders (FWO), Grant Number 1102424N.}, Thomas Witdouck \thanks{KU Leuven Kulak Kortrijk Campus, Department of Mathematics. Email: thomas.witdouck@kuleuven.be. The author was supported by a PhD fellowship of the Research Fund - Flanders (FWO), Grant Number 1153122N.}}
\date{\today}

\begin{document}
\maketitle

\begin{abstract}
	Let $G$ be a group and $\varphi$ an automorphism of $G$. Two elements $x,y \in G$ are said to be $\varphi$-conjugate if there exists a third element $z \in G$ such that $z x \varphi(z)^{-1} = y$. Being $\varphi$-conjugate defines an equivalence relation on $G$. The group $G$ is said to have the $R_\infty$--property if all its automorphisms $\varphi$ have infinitely many $\varphi$-conjugacy classes. For finitely generated torsion-free nilpotent groups, the so-called Mal'cev completion of the group is a useful tool in studying this property. Two groups have isomorphic Mal'cev completions if and only if they are abstractly commensurable. This raises the question whether the $R_\infty$--property is invariant under abstract commensurability within the class of finitely generated torsion-free nilpotent groups. We show that the answer to this question is negative and provide counterexamples within a class of 2-step nilpotent groups associated to edge-weighted graphs. These groups are commensurable to 2-step nilpotent quotients of right-angled Artin groups.
\end{abstract}

\section{Introduction}

Let $G$ and $H$ be groups. We say $G$ is \textit{abstractly commensurable with} $H$ if there exist finite index subgroups $G' \subset G$ and $H' \subset H$ such that $G'$ is isomorphic to $H'$. The starting point of this paper is the following question:

\begin{question}
	Is the $R_\infty$--property for groups an abstract commensurability invariant?
\end{question}

It is known however that the answer to this question is negative. The easiest counterexample is given by the fundamental group of the Klein bottle. This group has the $R_\infty$--property, while it has a subgroup of index 2 isomorphic to $\Z^2$ which is known not to have the $R_\infty$--property, see \cite[Section 2]{gw09-1}. Note that the fundamental group of the Klein bottle is torsion-free virtually abelian, but not nilpotent. A new interesting question arises when we ask our groups to be finitely generated torsion-free nilpotent.
\begin{question}
    \label{question2}
	Is the $R_\infty$--property for groups an abstract commensurability invariant within the class of finitely generated torsion-free nilpotent groups?
\end{question}

The reason for considering finitely generated torsion-free nilpotent groups is that there is a well-known abstract commensurability invariant within this class of groups, namely the \textit{rational Mal'cev completion} (see section \ref{sec:LieAlgAssToGrps}). Two finitely generated torsion-free nilpotent groups are abstractly commensurable if and only if they have isomorphic rational Mal'cev completions. In an attempt to answer Question \ref{question2}, we arrived at the following observation. We recall that a linear self-map on a finite-dimensional vector space is called \textit{integer-like} if its characteristic polynomial has integral coefficients and constant term equal to $\pm 1$.

\begin{proposition}
    \label{prop:RinftyAndCommensurability}
	Let $G$ be a finitely generated torsion-free nilpotent group. All groups that are abstractly commensurable to $G$ have the $R_\infty$--property if and only if every integer-like automorphism on the associated Mal'cev rational Lie algebra has an eigenvalue 1.
\end{proposition}

The actual answer to Question 2 turns out to be negative and we will present a counterexample within the class of 2-step nilpotent quotients of right-angled Artin groups. Using weight functions on the edges of the defining graphs, one can define groups which are commensurable to these 2-step nilpotent quotients of right-angled Artin groups, but which have considerably less induced automorphisms on the abelianization. These groups are defined in section \ref{sec:weightedRAAG} and the restrictions on their automorphisms are proven in section \ref{sec:automorpismsOfWeightedRAAG}. The (to our knowledge) smallest concrete counterexample to Question \ref{question2} is then presented in \ref{ex:mainCounterexample}, which proves the following. 

\begin{theorem}
    There exist finitely generated torsion-free 2-step nilpotent groups which do not have the $R_\infty$--property but have for any $n \in \N \setminus \{0,1\}$ a subgroup of index $n$ which does have the $R_\infty$--property.
\end{theorem}

\section{Preliminaries}

First, let us fix some notation. For any group $G$ we define the \textit{lower central series} of $G$ inductively by setting $\gamma_1(G)=G$ and $\gamma_{i+1}(G)=[\gamma_i(G),G]$ (for all $i\in \N_0$). For a Lie algebra $L$, we use the same notation for the lower central series, i.e. $\gamma_1(L) = L$ and $\gamma_{i+1}(L) = [L, \gamma_i(L)]$ (for all $i\in \N_0$).

\subsection{Lie algebras associated to groups}
\label{sec:LieAlgAssToGrps}
In this section, we recall two well-known constructions to associate a Lie algebra to a group and how they relate to one another.


\paragraph{The Lie ring construction} Let $G$ be a group. Denote the factors of the lower central series of $G$ by $L_i(G):=\gamma_i(G)/\gamma_{i+1}(G)$. The quotients are abelian (and thus $\Z$-modules). We define the Lie ring $L(G)$ as the direct sum of these $\Z$-modules, i.e.
\[ L(G):=\bigoplus_{i=1}^{\infty} L_i(G)\]
and the Lie bracket on $L(G)$ is induced by setting (for $i,j\in \N_0$, $g\in \gamma_i(G)$ and $h\in \gamma_j(G)$)
\[ [g\gamma_{i+1}(G),h\gamma_{j+1}(G)]:=[g,h]\gamma_{i+j+1}(G).\]
In order to turn $L(G)$ into a Lie algebra over a field $K$, it suffices to tensor $L(G)$ with $K$. We define $L^K(G):=L(G)\otimes_{\Z} K=\oplus_{i=1}^{\infty}L_i^K(G)$ where $L_i^K(G):=L_i(G)\otimes_{\Z}K$ (for all $i\in \N_0$). Moreover, $L^K(G)$ is clearly $\N_0$-graded (by using the $L_i^K(G)$) and if $G$ is finitely generated, the $L_i^K(G)$ are finite dimensional. The lower central series of $L^K(G)$ is given by
\[ \gamma_i(L^K(G))=\bigoplus_{j=i}^{\infty} L_j^K(G)\]
for all $i\in \N_0$.\\
Note that this way of associating a Lie algebra to a group is functorial. In particular, any automorphism of the group $G$ induces a graded Lie algebra automorphism of $L^K(G)$  (i.e. an automorphism which leaves the subspaces $L_i^K(G)$ invariant). Indeed, fix an automorphism $\varphi\in \Aut(G)$. Since the lower central series of $G$ is characteristic, we obtain induced automorphisms $\varphi_i\in \Aut(L_i(G))$ (for all $i\in \N_0$). These $\Z$-module automorphisms extend to automorphisms $\overline{\varphi_i}:=\varphi_i\otimes \Id$ of $L_i^K(G)$. Combining them yields a graded Lie algebra automorphism $\overline{\varphi}\in \Aut_g(L^K(G))$, i.e. we set $\overline{\varphi}(x)=\overline{\varphi_i}(x)$ for each $x\in L_i^K(G)$.

\paragraph{The Mal'cev completion} A group $G$ is said to be \textit{radicable} if the assignment $x \mapsto x^n$ defines a bijection on $G$ for any non-zero integer $n$. If $G$ is a finitely generated torsion-free nilpotent group, there exists a radicable group $G^\Q$ of which $G$ is a subgroup, such that for any $x \in G^\Q$ there is a positive integer $n > 0$ such that $x^n \in G$. The group $G^\Q$ is nilpotent and unique, in the sense that, if there exists another group with the same properties, it is isomorphic to $G^\Q$ with the isomorphism restricting to the identity on $G$. The group $G^\Q$ is called the \textit{rational Mal'cev completion} of $G$.\\
The group $G^\Q$ can also be given the structure of a nilpotent Lie algebra over $\Q$ by use of the Baker-Campbell-Hausdorff formula. We will refer to this Lie algebra as the \textit{Mal'cev Lie algebra} of $G$ and write it as $M(G)$. Note that $G^\Q$ and $M(G)$ have the same underlying sets. It is well-known that two finitely generated torsion-free nilpotent groups have isomorphic Mal'cev completions (or, equivalently, isomorphic Mal'cev Lie algebras) if and only if they are abstractly commensurable. Moreover, if $G$ and $H$ are finitely generated torsion-free nilpotent groups with isomorphic finite index subgroups $G' \subset G$ and $H'\subset H$, then there is a unique isomorphism between $M(G)$ and $M(H)$ which restricts to the given isomorphism between $G'$ and $H'$.\\
As with the graded Lie ring construction, the association of the rational Mal'cev completion to a finitely generated torsion-free nilpotent group is functorial. Any automorphism of $G$ extends uniquely to an automorphism of $G^\Q$ which in turn defines an automorphism of the \mbox{Lie algebra $M(G)$}.

Next, we recall a known result that relates $L^\Q(G)$ and $M(G)$ for a finitely generated torsion-free nilpotent group $G$. Let $L$ be any Lie algebra. Similar to the Lie ring construction from above, one can define from $L$ a new $\N_{0}$-graded Lie algebra $\gr(L)$ by
\[ \gr(L) = \bigoplus_{i = 1}^\infty \gamma_i(L)/\gamma_{i+1}(L) \]
with Lie bracket defined by $[ X + \gamma_{i+1}(L), \, Y + \gamma_{j+1}(L)] = [X, Y] + \gamma_{i + j + 1}(L)$ for $i,j\in \N_0$ and any $X \in \gamma_i(L)$ and $Y \in \gamma_j(L)$.\\
In \cite[Lemma 3.3]{quil68} it is proven that for any finitely generated torsion-free nilpotent group $G$, the injection of $G$ into its rational Mal'cev completion $G^\Q$ induces an isomorphism of graded Lie algebras

\begin{equation}
    \label{eq:isoMalcevGraded}
    L^\Q(G) \cong \gr(M(G)).
\end{equation}

From this result, it is not hard to see that, given an automorphism of $G$, the induced automorphisms on $L^\Q(G)$ and $M(G)$ have the same eigenvalues. Moreover, there is a very nice connection between the eigenvalues of these induced automorphisms and the Reidemeister number of the original automorphism.

\begin{lemma}[\cite{roma11-1},\cite{dg14-1}]\label{lem:eigenvalue1}
	Let $G$ be a finitely generated nilpotent group and $\varphi\in \Aut(G)$. Let $\overline{\varphi}\in \Aut_g(L^K(G))$ be the induced graded automorphism on $L^K(G)$ and $\varphi'\in \Aut(M(G))$ the induced automorphism on $M(G)$. Then
        \[R(\varphi)=\infty \quad \Longleftrightarrow\quad \overline{\varphi} \text{ has eigenvalue } 1\quad \Longleftrightarrow\quad \varphi' \text{ has eigenvalue } 1.\]
\end{lemma}

\subsection{(Nilpotent) partially commutative groups and Lie algebras}

In this section, we introduce groups and Lie algebras that will be used throughout this paper.


Let $\Gamma=(V,E)$ be a finite undirected simple graph with (finite) vertex set $V$ and edge set $E\subset \{\{v,w\}\mid v,w\in V \text{ and } v\neq w\}$. From now on, we always assume a graph $\Gamma$ to be finite, undirected and simple. The \textit{free partially commutative group} $A(\Gamma)$ associated to the graph $\Gamma$ is defined by
\[ A(\Gamma):=\langle V\mid [v,w]=1 \text{ for } v,w\in V \text{ with } \{v,w\}\not\in E \rangle.\]
These groups are also known as \textit{right-angled Artin groups} or \textit{graph groups}. Sometimes the complement graph is used to define $A(\Gamma)$ (i.e. edges become non-edges in our definition). Nevertheless, our convention will come in handy later when generalising to groups associated to graphs with weights on their edges. By taking the quotient of $A(\Gamma)$ with $\gamma_{c+1}(A(\Gamma))$ (with $c>1$) we get a finitely generated torsion-free c-step nilpotent group \[A(\Gamma,c):=A(\Gamma)/\gamma_{c+1}(A(\Gamma)).\]
We will mostly use the 2-step nilpotent quotient and thus for simplicity, we write \[G_{\Gamma}:=A(\Gamma,2).\]

Next to a graph $\Gamma$, now also fix a field $K$. We can associate, similarly as with $A(\Gamma)$, a Lie algebra $L^{K}(\Gamma)$ (over $K$) to the graph $\Gamma$. Denote with $\mathfrak{f}^K(V)$ the free Lie algebra on $V$ over $K$ and with $I(\Gamma)$ the ideal of $\mathfrak{f}^K(V)$ generated by the set of Lie brackets $\{[v,w]\mid v,w\in V,\: \{v,w\}\not\in E\}$. The Lie algebra $L^K(\Gamma)$, which is also known as the \textit{free partially commutative Lie algebra over $K$}, is defined as the quotient of $\mathfrak{f}^K(V)$ by $I(V)$. Similarly, we get the \textit{free $c$-step nilpotent partially commutative Lie algebra} $L^K(\Gamma,c)$ by taking the quotient of $L^K(\Gamma)$ with $\gamma_{c+1}(L^K(\Gamma))$. Summarising, we defined
\[ L^K(\Gamma):=\frac{\mathfrak{f}^K(V)}{I(\Gamma)}\quad \text{and}\quad L^K(\Gamma,c):=\frac{L^K(\Gamma)}{\gamma_{c + 1}(L^K(\Gamma))}.\]

The free Lie algebra $\mathfrak{f}^K(V)$ is canonically graded by defining inductively $\mathfrak{f}_1^K(V):= \Sp_K(V)$ (i.e. the vector space span of $V$ over $K$) and $\mathfrak{f}_{i+1}^K(V):=[\mathfrak{f}_1^K(V),\mathfrak{f}_i^K(V)]$ (for $i\in \N_0$). Since the ideal $I(\Gamma)$ is a graded ideal $\mathfrak{f}^K(V)$ it follows that $L(\Gamma, c)$ inherits an $\N_0$-grading from $\mathfrak{f}^K(V)$. We denote the subspaces that make up this grading by $L_i^K(\Gamma,c)$ ($ i \in \N_0$). Note that $L_i^K(\Gamma,c)=\{0\}$ for $i>c$.

A classical result from \cite[Theorem 2.1]{dk92} establishes a connection between $A(\Gamma)$ and $L^K(\Gamma)$ using the Lie ring construction of the previous section. Concretely, there is an isomorphism of graded Lie algebras
\begin{equation}
    \label{eq:isomorphismRAAGLieAlgGrp}
    L^K(\Gamma) \cong L^K(A(\Gamma))
\end{equation}
which is determined by taking the identity on the vertices $V$.

\begin{remark}\label{rem:RestrictionToAbelianization}
	Note that the subspace $\Sp_K(V)$ generates $L^K(\Gamma,c)$. In particular, any graded automorphism $\varphi\in \Aut_g(L^K(\Gamma,c))$ is completely determined by its restriction to $\Sp_K(V)$. So if we denote with $\pi:\Aut_g(L^K(\Gamma,2))\to \Gl(\Sp_K(V)):\varphi\mapsto \varphi|_{\Sp_K(V)}$, then $\pi$ is an isomorphism onto its image.
\end{remark}


\section{The \texorpdfstring{$R_\infty$}{Roo}-property and commensurability}
In this section, we prove Proposition \ref{prop:RinftyAndCommensurability} from the introduction and apply it to the class of free nilpotent partially commutative groups.
\begin{proof}[Proof of Proposition \ref{prop:RinftyAndCommensurability}]
    Let $G$ be a finitely generated torsion-free nilpotent group, $G^\Q$ its Mal'cev completion and $M(G)$ the Mal'cev Lie algebra.
    
    First, assume that every integer-like automorphism of $M(G)$ has an eigenvalue 1. Let $H$ be a group commensurable with $G$ and $\varphi$ any automorphism of $H$. Then $\varphi$ induces an automorphism on the Mal'cev Lie algebra of $H$, which is isomorphic to $M(G)$. Moreover, this automorphism must be integer-like, thus by the assumption it has an eigenvalue 1. Using Lemma \ref{lem:eigenvalue1}, we find that $R(\varphi) = \infty$ and thus that $H$ has the $R_\infty$--property.

    Conversely, assume that every group commensurable to $G$ has the $R_\infty$--property. Let $\varphi$ be an integer-like automorphism of $M(G)$. A property of such an automorphism is that there exists a basis $\{ x_1, \ldots, x_n\}$ of $M(G)$ with respect to which $\varphi$ has a matrix representation in $\Gl_n(\Z)$. Moreover, it follows from the BCH formula that there exists a positive integer $k$ big enough, such that
    \[ H := \Sp_\Z(\{ kx_1, \ldots, kx_n \}) \]
    is a subgroup of $G^\Q$. Note as well that $H$ is preserved by $\varphi$. Clearly, $H$ is torsion-free and nilpotent since it is a subgroup of $G^\Q$. To check that $H$ is finitely generated, it suffices to check that its abelianization is so, since $H$ is nilpotent. This follows straightforwardly from the fact that the abelianizations of $M(G)$ (seen as abelian group for '+') and $G^\Q$ are isomorphic. At last, using the BCH formula, one can easily verify that for any $x \in G^\Q$, there exists a non-zero integer $m$ such that $x^m$ lies in $H$. We thus have that $G^\Q$ is the Mal'cev completion of $H$. As a consequence, $H$ is commensurable to $G$. Thus, by assumption, $H$ has the $R_\infty$--property and $R(\varphi|_H) = \infty$. Lemma \ref{lem:eigenvalue1} then implies that the induced morphism on $M(H) = M(G)$, which must be equal to $\varphi$, has an eigenvalue 1. This completes the proof.
\end{proof}

It is not clear to us whether in the statement of Proposition \ref{prop:RinftyAndCommensurability} one can replace `the associated Mal'cev Lie algebra' with `the associated graded Lie algebra'. Equivalently, we do not have an answer to the following question.

\begin{question}
    Does there exist a finite-dimensional nilpotent rational Lie algebra $L$ for which every integer-like automorphism has an eigenvalue 1, but for which $\gr(L)$ does not have this property?
\end{question}

It is not hard to check that the other direction is true, namely if $\gr(L)$ has the property then so does $L$.

\subsection{Free abelian factors and \texorpdfstring{$R_\infty$}{Roo}-property for commensurable groups}
In light of Question \ref{question2}, we argue in this section that we can forget about the (maximal) free abelian factors when stating the question if we are considering 2-step nilpotent groups, i.e. we prove that if $H$ is finitely generated torsion-free 2-step nilpotent and $H^\Q$ does not have an abelian direct factor, then the following are equivalent for any $m\in \N_0$
\begin{center}
		None of the (fg. tf. 2-step nilpotent) groups commensurable with $H$ has the $R_{\infty}$--property\\
		\vspace{2pt}$\Updownarrow$\vspace{2pt}\\
		None of the (fg. tf. 2-step nilpotent) groups commensurable with $\Z^m\times H$ has the $R_{\infty}$--property.
\end{center}

Recall that the \textit{adapted lower central series} of a group $G$ is defined as the series of isolators of the terms of the lower central series, i.e. the series of terms $\sqrt[G]{\gamma_i(G)}=\{g\in G \mid \exists\, l\in \N_0:\: g^l\in \gamma_i(G)\}$ (for $i\in \N_0$). If $G$ is nilpotent, these are in fact characteristic subgroups in $G$. We will sometimes omit the superscript $G$ if the group is clear. Note that $H^\Q$ does not have an abelian direct factor is equivalent with $Z(H)=\sqrt[H]{\gamma_2(H)}$ if we are considering nilpotency class 2.

\begin{lemma}\label{lem:GQ abelian factors}
	Let $G$ be a finitely generated torsion-free 2-step nilpotent group, then $Z(G)=\sqrt{\gamma_2(G)}$ if and only if $G^{\Q}$ does not have an abelian direct factor.
\end{lemma}
\begin{proof}
	If $G^{\Q}$ does not have an abelian factor, then Proposition 2.18 in \cite{send22} tells us that $Z(G)=\sqrt{\gamma_2(G)}$. Hence, it suffices to prove the converse.\\
	Suppose that $Z(G)=\sqrt{\gamma_2(G)}$ and $G^{\Q}\cong \Q^l\times G'$ with $l\in \N$. Since $\gamma_3(G')=\gamma_3(G^{\Q})=1$, it follows that $G'$ is finitely generated and nilpotent of class at most 2. Thus we obtain that
	\begin{align*}
		\dim_{\Q}(Z(G'))&\leq l+\dim_{\Q}(Z(G'))=\dim_{\Q}(Z(G^{\Q}))=\dim_{\Q}((Z(G))^{\Q})\\
		&=\dim_{\Q}((\sqrt{\gamma_2(G)})^{\Q})=\dim_{\Q}(\gamma_2(G^{\Q}))=\dim_{\Q}(\gamma_2(G'))\\
		&\leq \dim_{\Q}(Z(G')).
	\end{align*}
	In particular, this implies that $l=0$ and thus $G^{\Q}$ does not have an abelian direct factor if $Z(G)=\sqrt{\gamma_2(G)}$.
\end{proof}

Next, we argue that for a finitely generated torsion-free 2-step nilpotent group its finite index subgroups are the only groups commensurable with it (up to isomorphism). This follows directly by applying the next lemma.
\begin{lemma}\label{lem:finite index 2-step nilpotent}
	Let $G$ be a finitely generated torsion-free 2-step nilpotent group and $H\subset_{\text{fin}} G$. Then there exists some $G'\subset_{\text{fin}}H$ such that $G'\cong G$.
\end{lemma}
\begin{proof}
	Fix an adapted set of generators of $G$ $\{x_1,\dots,x_n,y_1,\dots,y_m\}$, i.e. such that $\{x_1\sqrt{\gamma_2(G)},\dots,x_n\sqrt{\gamma_2(G)}\}$ (respectively $\{y_1,\dots,y_m\}$) is a $\Z$-basis for $G/\sqrt{\gamma_2(G)}\cong \Z^n$ (respectively $\sqrt{\gamma_2(G)}\cong \Z^m$). Hence, there are $c_{ij,l}\in \Z$ (for $l=1,\dots,m$) such that
	\[ G=\Bigg\langle x_1,\dots,x_n,y_1,\dots,y_m\:\Biggg\vert \:   
	\begin{array}{l}
		[x_i,x_j]=\prod_{l=1}^my_l^{c_{ij,l}};\\
		\relax [x_i,y_{l_1}]=[y_{l_1},y_{l_2}]=1 \forall i,l_1,l_2
	\end{array}
	\Bigg \rangle.\]
	Since $[G:H] <\infty$ there are $a_i,b_l\in \N_0$ (for all $i,l$) such that $x_i^{a_i},y_l^{b_l}\in H$. Define $k:=\text{lcm}_{i,l}\{a_i,b_l\}$ and consider $G':=\langle x_1^k,\dots,x_n^k,y_1^{k^2},\dots,y_m^{k^2}\rangle\subset G$. Since $G$ is 2-step nilpotent, it follows that $G'\cong G$. Note that $x_i^k,y_l^{k^2}\in H$ (for all $i,l$) and thus by Lemma 2.8 in \cite{baum71-1} we obtain that $G'\subset_{\text{fin}} G$. Hence, we can conclude that $G'\subset_{\text{fin}}H$ and $G'\cong G$.
\end{proof}

\begin{corollary}\label{cor:commensurablity 2-step nilpotent}
	If $G$ and $H$ are two finitely generated torsion-free 2-step nilpotent groups that are commensurable, then there is some finite index subgroup $N\subseteq_{\text{fin}} G$ such that $H\cong N$.
\end{corollary}

Next we argue that for a finitely generated torsion-free 2-step nilpotent group, we can always factor out the (maximal) free abelian factor.

\begin{lemma}\label{lem:direct product 2-step nilpotent}
	Let $H$ be a finitely generated torsion-free 2-step nilpotent group, then $H\cong\Z^m\times H'$ (with $m\in \N$) such that
	\begin{enumerate}[label = (\roman*)]
		\item $H'$ is a finitely generated torsion-free 2-step nilpotent group
		\item $Z(H')=\sqrt[H']{\gamma_2(H')}$
	\end{enumerate}
\end{lemma}

\begin{proof}
    Fix a $\Z$-basis $\{x_1Z(H),\dots,x_nZ(H)\}$ (respectively $\{y_1\sqrt{\gamma_2(H)},\dots, y_m\sqrt{\gamma_2(H)}\}$) of $H/Z(H)\cong \Z^n$ (respectively $Z(H)/\sqrt{\gamma_2(H)}\cong \Z^m$). Define $H':=\langle x_1,\dots,x_n\rangle\sqrt{\gamma_2(H)}$. One can easily check that any element of $H'$ can be uniquely written as $\prod_{i=1}^n x_i^{t_i}\gamma$ (with $t_i\in \Z$ and $\gamma\in \sqrt{\gamma_2(H)}$) and thus $H\cong\Z^m\times H'$.\\
	Now we prove that $\sqrt[H']{\gamma_2(H')}=\sqrt[H]{\gamma_2(H)}=Z(H')$. Note that $\gamma_2(H)=\gamma_2(H')$ and thus by using that $\sqrt[H]{\gamma_2(H)}\subset H'\subset H$ it readily follows that $\sqrt[H']{\gamma_2(H')}=\sqrt[H]{\gamma_2(H)}$. Since $\sqrt[H]{\gamma_2(H)}\subset Z(H)\cap H'\subset Z(H')$, it suffices to prove the other two inclusions. Since $H/Z(H)=\langle x_1Z(H),\dots,x_nZ(H)\rangle$ and since $\langle x_1,\dots,x_n\rangle\subset H'$, it follows that $Z(H')=Z(H)\cap H'$. Now fix any $h'=\prod_{i=1}^n x_i^{t_i}\gamma\in Z(H)\cap H'$ (where $\gamma\in \sqrt[H]{\gamma_2(H)}\subset Z(H)$ and $t_i\in \Z$). Since $\gamma\in \sqrt[H]{\gamma_2(H)}\subset Z(H)$, we get that $\prod_{i=1}^n x_i^{t_i'}\in Z(H)$ and thus $t_i=0$ for all $i$. Hence, $h'\in \sqrt[H]{\gamma_2(H)}$ and thus we conclude that $\sqrt[H']{\gamma_2(H')}=\sqrt[H]{\gamma_2(H)}=Z(H')=Z(H)\cap H'$.
\end{proof}


Any finitely generated group $G$ is polycyclic, i.e. it has a series $1=G_1\lhd G_2\lhd \dots\lhd G_n=G$ with all $G_{i+1}/G_i$ cyclic. The number of infinite cyclic factors in such a series is called the \textit{Hirsch number} $h(G)$ of $G$. The next Lemma gives some important properties.
\begin{lemma}\label{lem:Hirsch number}
    Let $G$ be a finitely generated nilpotent group. Then the Hirsch number $h(G)$ of $G$ is well-defined, i.e. it does not depend on the chosen series. Moreover, if $H\subset G$ and $N\lhd G$, then it holds that
    \begin{enumerate}[label = (\roman*)]
        \item $h(G)\geq h(H)$
        \item $h(G)=h(H) \: \Leftrightarrow\: [G:H]<\infty$
        \item $h(G)=h(N)+h(G/N)$
        \item $h(G)=0\: \Leftrightarrow\: |G|<\infty$
        \item if $G$ and $H$ are commensurable, then $h(G)=h(H)$.
    \end{enumerate}
\end{lemma}

\begin{remark}
    Remark that for a finitely generated torsion-free 2-step nilpotent group $H$, it holds that
    \[h(Z(H))=h(\sqrt[H]{\gamma_2(H)})\quad \Longleftrightarrow\quad Z(H)=\sqrt[H]{\gamma_2(H)}.\]
    Indeed, since $H$ is torsion-free 2-step nilpotent, it holds that $\sqrt[H]{\gamma_2(H)}\subset Z(H)$. If moreover $h(Z(H))=h(\sqrt[H]{\gamma_2(H)})$, then by Lemma \ref{lem:Hirsch number} it follows that $\sqrt[H]{\gamma_2(H)}\subset_{\text{fin}} Z(H)$. By Lemma 2.8 in \cite{baum71-1} it holds that $\gamma_2(H)\subset_{\text{fin}} \sqrt[H]{\gamma_2(H)}$. Thus we obtain that $\gamma_2(H)\subset_{\text{fin}}Z(H)$ and thus for any $g\in Z(H)$ there is some $k\in \N_0$ such that $g^k\in \gamma_2(H)$. In particular, it holds that $Z(H)\subset \sqrt[H]{\gamma_2(H)}$ and thus equality follows.
\end{remark}

Now we have gathered all the tools to prove the result mentioned earlier.
\begin{theorem}\label{thm:Z^k times H}
	Let $H$ be a finitely generated torsion-free 2-step nilpotent group with $h(Z(H))=h(\sqrt[H]{\gamma_2(H)})$ (or equivalently $Z(H)=\sqrt[H]{\gamma_2(H)}$) and $m\in \N_0$. Then
    \begin{center}
		None of the (fg. tv. 2-step nilpotent) groups commensurable with $H$ has the $R_{\infty}$--property\\
		\vspace{2pt}$\Updownarrow$\vspace{2pt}\\
		None of the (fg. tv. 2-step nilpotent) groups commensurable with $\Z^m\times H$ has the $R_{\infty}$--property.
	\end{center}
\end{theorem}
\begin{proof}
	We start by proving the implication from top to bottom. Fix any finitely generated torsion-free 2-step nilpotent group $N$ which is commensurable with $G:=\Z^m\times H$. By Corollary \ref{cor:commensurablity 2-step nilpotent} we can assume that $N\subset_{\text{fin}}G$. Using Lemma \ref{lem:direct product 2-step nilpotent}, we can write $N\cong \Z^n\times N'$ (with $n\in \N$ such that $Z(N)\cong \Z^n\times \sqrt[N]{\gamma_2(N)}$). We argue that $n=m$. One can check that $\sqrt[G]{\gamma_2(G)}=\sqrt[H]{\gamma_2(H)}$. Since two finitely generated torsion-free nilpotent groups are commensurable if and only if their Mal'cev completions are isomorphic, we know that $\sqrt[N]{\gamma_2(N)}$ and $\sqrt[G]{\gamma_2(G)}$ as well as $Z(N)$ and $Z(G)$ are commensurable. Thus, using Lemma \ref{lem:Hirsch number}, we obtain that
        \begin{align*}
            n&=h\Bigg(\frac{Z(N)}{\sqrt[N]{\gamma_2(N)}}\Bigg)=h(Z(N))-h(\sqrt[N]{\gamma_2(N)})\\&=h(Z(G))-h(\sqrt[G]{\gamma_2(G)})=m+h(Z(H))-h(\sqrt[H]{\gamma_2(H)})\\&=m.
        \end{align*}
	It follows that $N\cong \Z^m\times N'$ and thus (by \cite[Proposition 5]{bmo16}) we have that
	\[ \Q^m\times (N')^{\Q} \cong N^{\Q}\cong G^{\Q}\cong \Q^m\times H^{\Q}.\]
	Hence, the associated Lie algebras $\Q^m\times M(N')$ and $\Q^m\times M(H)$ are isomorphic. By Proposition 2.5 in \cite{dere19} it follows that $M(N')\cong M(H)$, thus $(N')^{\Q}\cong H^{\Q}$, concluding that $N'$ and $H$ are commensurable. In particular, $N'$ doesn't have the $R_{\infty}$--property. Hence, we can fix automorphisms $\varphi_{N'}\in \Aut(N')$ and $\varphi_{\Z^m}\in \Aut(\Z^m)$ with $R(\varphi_{N'})<\infty$ and $R(\varphi_{\Z^m})<\infty$. By for example \cite[Proposition 2.4]{sen21-2} we can now conclude that $R(\varphi_{\Z^m}\times \varphi_{N'})=R(\varphi_{\Z^m})R(\varphi_{N'})<\infty$. Hence, $N$ doesn't have the $R_{\infty}$--property.\\
	
    Now we prove the other implication. Assume that $N'$ is a finitely generated torsion-free 2-step nilpotent group which is commensurable with $H$ and which has the $R_{\infty}$--property. Define $N:=\Z^m\times N'$. Since $H^{\Q}\cong (N')^{\Q}$ and $Z(H)=\sqrt[H]{\gamma_2(H)}$ it follows by Lemma \ref{lem:GQ abelian factors} that $Z(N')=\sqrt[N']{\gamma_2(N')}$ and $(N')^{\Q}$ doesn't have an abelian factor. Theorem 3.8 in \cite{send22} now tells us that $\Spec(N)=\Spec(\Z^m)\cdot \Spec(N')$. In particular, we can conclude that $N$ has the $R_{\infty}$--property which concludes the proof.
\end{proof}
If $\Gamma(V,E)$ is a graph, then we use $V_{\text{iso}} \subset V$ to denote the set of \textit{isolated vertices}, i.e. vertices that are not adjacent to any other vertex. Now we can apply Theorem \ref{thm:Z^k times H} to our set-up of groups associated to graphs.
\begin{corollary}\label{cor:R-infinity isolated vertices}
	Let $\Gamma(V,E)$ be a finite undirected simple graph. Denote with $\tilde{\Gamma}$ the subgraph of $\Gamma$ on the vertex set $V\setminus V_{\text{iso}}$. Then $G_{\Gamma}=\Z^{|V_{\text{iso}}|}\times G_{\tilde{\Gamma}}$ and $Z(G_{\tilde{\Gamma}})=\sqrt{\gamma_2(G_{\tilde{\Gamma}})}$. In particular,
	\begin{center}
		None of the (fg. tv. 2-step nilpotent) groups commensurable with $G_{\tilde{\Gamma}}$ has the $R_{\infty}$--property\\
		\vspace{2pt}$\Updownarrow$\vspace{2pt}\\
		None of the (fg. tv. 2-step nilpotent) groups commensurable with $G_{\Gamma}$ has the $R_{\infty}$--property.
	\end{center}
\end{corollary}

\begin{remark}
    In an upcoming paper (see \cite{dl23-2}), we argue that all finitely generated torsion-free 2-step nilpotent groups with Hirsch number at most six do not have the $R_\infty$--property. Note that for all graphs $\Gamma(V,E)$ it holds that $h(G_\Gamma)=|V|+|E|$. Hence, if $|V|+|E|\leq 6$ then none of the groups commensurable with $G_\Gamma$ has the $R_\infty$--property. Thus applying Corollary \ref{cor:R-infinity isolated vertices} yields the stronger result that
    \begin{center}
        if $\Gamma(V,E)$ is a finite undirected simple graph such that $|V\setminus V_{\text{iso}}|+|E|\leq 6$, then none of the groups commensurable with $G_{\Gamma}$ has the $R_{\infty}$--property.
    \end{center}
    For example, this holds for the graph depicted below. Note that for this graph $G_\Gamma\cong \Z\times H_1\times H_1$, where $H_1$ is the Heisenberg group.
    
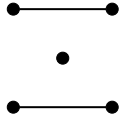
\begin{figure}[h]
        \centering
        \begin{tikzpicture}[main_node/.style={circle,draw,fill,minimum size=4.5pt,inner sep=0, outer sep=0]}]
				\def\lengte{1.3cm}
				\node[main_node] (0) at (0,\lengte) {};
				\node[main_node] (1) at (\lengte,\lengte) {};
				\node[main_node] (2) at (\lengte,0) {};
				\node[main_node] (3) at (0,0) {};
				\node[main_node] (4) at ({0.5*\lengte},{0.5*\lengte}) {};
				
				\path[draw, line width=0.75pt]
				(0) edge 	node {} (1) 
				(2) edge 	node {} (3)
				;
	\end{tikzpicture}
    \end{figure}
\end{remark}

\subsection{The case of free nilpotent partially commutative groups}

In the final section of this paper, a result from \cite{witd23} is recalled (as Theorem \ref{thm:LowerUpperBoundNilpIndex}) which tells us for certain $c>1$ whether $A(\Gamma, c)$ has the $R_\infty$--property or not. In this section, we explore how $R_\infty$ and abstract commensurability behave within the class of free nilpotent partially commutative groups. As it turns out, the free nilpotent partially commutative group $A(\Gamma,c)$ (with $c>1$) contains a lot of information in the following sense: if $A(\Gamma,c)$ has the $R_{\infty}$--property, then all finitely generated torsion-free nilpotent groups commensurable with it must have the $R_{\infty}$-property too. In fact, even a wider family of groups are implied to have $R_\infty$, which is stated in the following proposition.

\begin{proposition}
	\label{prop:isoGradLieAlgImplicationRinfty}
	Let $\Gamma$ be a graph and $c > 1$ an integer. The following are equivalent:
    \begin{enumerate}[label = (\roman*)]
        \item \label{item:equivRAAG1} $A(\Gamma, c)$ has the $R_\infty$--property.
        \item \label{item:equivRAAG2} Every integer-like automorphism of $L^\Q(\Gamma, c)$ has an eigenvalue 1.
        \item \label{item:equivRAAG3} Every group $H$ with $L^\Q(H) \cong L^\Q(\Gamma, c)$ has the $R_\infty$--property.
    \end{enumerate}
\end{proposition}
\begin{proof}
    $(i) \Rightarrow (ii)$. Take any integer-like automorphism $\varphi$ of $L(\Gamma, c)$. By Corollary 3.10 from \cite{witd23}, there exists an automorphism $\phi$ on $A(\Gamma, c)$ such that the induced automorphism $\overline{\phi}$ on $L^\Q(\Gamma, c)$ has the same characteristic polynomial as $\varphi$. Using Lemma \ref{lem:eigenvalue1} and the assumption, the automorphism $\overline{\phi}$ has an eigenvalue 1 and hence $\varphi$ has an eigenvalue 1.

    $(ii) \Rightarrow (iii)$. Let us write $\tilde{H}$ for the group $H/\gamma_{c+1}(H)$. Clearly, we have the isomorphisms
	\[ L^\Q(\tilde{H}) \cong \frac{L^\Q(H)}{\gamma_{c+1}(L^\Q(H))} \cong \frac{L^\Q(\Gamma, c)}{\gamma_{c+1}(L^\Q(\Gamma, c))} \cong L^\Q(\Gamma, c). \]
    Let $\varphi$ be any automorphism of $\tilde{H}$ and write $\overline{\varphi}$ for the induced graded automorphism on $L^\Q(\Gamma, c)$ where we use the identification above. Note that $\overline{\varphi}$ is integer-like. Thus, by assumption, it has an eigenvalue 1. Lemma \ref{lem:eigenvalue1} then implies that $R(\varphi) = \infty$. Since $\varphi$ was an arbitrarily chosen automorphism of $\tilde{H}$, this proves that $\tilde{H}$ has the $R_\infty$--property and by Lemma 1.1. from \cite{gw09-1}, so does $H$.

    $(iii) \Rightarrow (i)$. This follows trivially if we substitute $A(\Gamma, c)$ for $H$ in the assumption. 
\end{proof}

\begin{remark}
    Note that any finitely generated torsion-free nilpotent group $H$, which is abstractly commensurable to $A(\Gamma, c)$, satisfies \ref{item:equivRAAG3} of the above proposition. Indeed, we get that $M(H) \cong M(A(\Gamma, c))$ and thus also $L^\Q(H) \cong \gr(M(H)) \cong \gr(M(A(\Gamma, c))) \cong L^\Q(A(\Gamma, c)) \cong L^\Q(\Gamma, c)$.
\end{remark}

\begin{remark}
	Note that in \ref{item:equivRAAG3} of the above proposition one can, in general, not replace the field $\Q$ with a non-trivial field extension of it (which would be a weaker assumption). Let us give a concrete example. Consider the 2-step nilpotent group $H$ with presentation
	\begin{equation}
		H = \left\langle \begin{array}{l}
			x_1, x_2, x_3, x_4,\\
			y_1, y_2, y_3
		\end{array} \middle\vert 
		\begin{array}{llll}
			\,[x_1, x_3] = y_1, 	& [x_1, x_4] = y_2, & [x_3, x_4] = y_3, &[y_i, y_j] = 1,\\
			\,[x_2, x_4] = y_1^2, 	& [x_2, x_3] = y_2, & [x_1, x_2] = 1
		\end{array} 
		\right\rangle.
	\end{equation}
	This group does not have the $R_\infty$--property since the automorphism of $H$ that is defined by
	\[ \begin{array}{ll}
		x_1 \mapsto x_1^{-1} x_2 		\quad\quad& y_1 \mapsto y_1^3 y_2^{-2}\\
		x_2 \mapsto x_1^{2} x_2^{-1} 	\quad\quad& y_2 \mapsto y_1^{-4} y_2^{3}\\
		x_3 \mapsto x_3^{-1} x_4 		\quad\quad& y_3 \mapsto y_3^{-1}\\
		x_4 \mapsto x_3^{2} x_4^{-1} 	\quad\quad&
	\end{array} \]
	induces an automorphism on $L^\C(H)$ which does not have 1 as an eigenvalue (see Lemma \ref{lem:eigenvalue1}). On the other hand, consider the graph $\Gamma$ as drawn below.
	
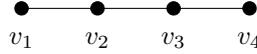
\begin{figure}[H]
		\centering
		\begin{tikzpicture}
			\draw (-1.5, 0) -- (-0.5, 0);
			\draw (-0.5, 0) -- (0.5, 0);
			\draw (0.5, 0) -- (1.5, 0);
			\filldraw [black] (-1.5, 0) circle (2.5pt) node[below = 0.2] {$v_1$};
			\filldraw [black] (-0.5, 0) circle (2.5pt) node[below = 0.2] {$v_2$};
			\filldraw [black] (0.5, 0) circle (2.5pt) node[below = 0.2] {$v_3$};
			\filldraw [black] (1.5, 0) circle (2.5pt) node[below = 0.2] {$v_4$};
		\end{tikzpicture}
	\end{figure}
	By Theorem 1.4 in \cite{witd23}, $G_{\Gamma}=A(\Gamma,2)$ has the $R_\infty$--property. However, over the field $K = \Q(\sqrt{2})$, we have a graded Lie algebra isomorphism $L^K(H) \cong L^K(\Gamma, 2)$ which is determined by
	\[ \begin{array}{ll}
		x_1 \mapsto v_1 + v_4 & x_3 \mapsto v_2 + v_3\\
		x_2 \mapsto \sqrt{2}(v_1 - v_4) \quad& x_4 \mapsto \sqrt{2}(v_2 - v_3).
	\end{array} \]
\end{remark}

\newpage
\section{Groups associated to edge-weighted graphs}
\label{sec:weightedRAAG}
\subsection{Relations on the vertices and edges}
\label{sec:relationsOnVerticesAndEdges}
In this section, we introduce notations and definitions used throughout the rest of the paper. For a finite undirected simple graph $\Gamma=(V,E)$ we will first define preorders on $V$ and $E$ (denoted with $\prec$) and induced equivalence relations on $V$ and $E$ (denoted with $\sim$), for which the sets of equivalence classes will be denoted by $\Lambda$ and $M$, respectively. Next, we define the quotient graph $\olGamma=(\Lambda,\olE,\Psi)$ associated to $\Gamma$. At last, we define partial orders on $\Lambda$ and $M$ (denoted with $\olprec$) and refine these orders to get total orders on $V$ and $E$ (denoted with $<$). All definitions are illustrated by an example at the end of the section.

Let $\Gamma=(V,E)$ be a finite undirected simple graph. Let $k:E\to \N_0$ be a map associating to each edge $e\in E$ a weight $k(e)\in \N_0$. We use $\Gamma(k)$ to denote the graph $\Gamma$ with weight function $k$ on its edges. From now on, we abbreviate $\Gamma(1)$ to $\Gamma$ (where $1:E\to \N_0:e\mapsto 1$).\\
Let $v\in V$ be a vertex of the graph $\Gamma$. We define its \textit{open} and \textit{closed neighbourhood} ($\Omega'(v)$ and $\Omega(v)$ respectively) by setting
\[ \Omega'(v):=\{w\in V\mid \{v,w\}\in E\} \quad \text{and} \quad \Omega(v):=\Omega'(v)\cup \{v\}. \]
Next, we define a preorder $\prec$ (i.e. a reflexive and transitive relation) on the vertex set $V$ by setting for any $v,w\in V$:
\[v\prec w \quad \Leftrightarrow \quad \Omega'(v)\subseteq \Omega(w).\]
This preorder on $V$ induces a preorder $\prec$ on $E$ by defining for all $\{v,w\}, \{v',w'\}\in E$ that
\[\{v,w\}\prec \{v',w'\}\quad \Longleftrightarrow\quad (v\prec v' \text{ and } w\prec w')\text{ or } (v\prec w' \text{ and } w\prec v').\]
Using these preorders, we define an equivalence relation $\sim$ on $V$ and $\sim$ on $E$ by setting for all $v,w\in V$ and $e,e'\in E$ that
\begin{align*}
    v\sim w &\quad\Longleftrightarrow\quad v\prec w \text{ and } w\prec v\\
    e\sim e' &\quad\Longleftrightarrow\quad e\prec e' \text{ and } e'\prec e.
\end{align*}
Note that for two vertices $v,w\in V$ the expression $v\sim w$ is equivalent to stating that the transposition $(v\: w)$ is an automorphism of the graph $\Gamma$. If none of the transpositions is a graph automorphism of $\Gamma$ (or equivalently, no two vertices are equivalent), we say that $\Gamma$ is \textit{transposition-free}. We let $[v]$ and $[e]$ denote the equivalence classes of $v\in V$ and $e\in E$, respectively. The equivalence classes on $V$ are called the \textit{coherent components} of $\Gamma$ and we denote the quotient space with $\Lambda:=V/\sim$. The correspondent quotient space $E/\sim$ is written as $M$. The map $\Psi:\Lambda\to \N_0$ sends any coherent component to the number of vertices it contains. Sometimes we use the abbreviation $\vert\lambda\vert$ instead of $\Psi(\lambda)$. Remark that if two vertices are connected via an edge, then all the edges between their corresponding coherent components are contained in $E$, i.e. for all (not necessarily distinct) $[v],[w]\in \Lambda$ it holds that
\[ \{v,w\}\in E\quad \Longrightarrow\quad \{v',w'\}\in E\text{ for \textbf{all} } v'\in [v] \text{ and }w'\in [w].\]
Hence, the graph $\Gamma$ induces unambiguously a graph $\olGamma$ with vertex set $\Lambda$ and edge set $\olE$ defined by
\[ \olE:=\{\{[v],[w]\}\mid \{v,w\}\in E\}. \]
The triple $\olGamma=(\Lambda,\olE,\Psi)$ is called the \textit{quotient graph} of $\Gamma$. Each edge of the quotient graph corresponds with a collection of edges of $E$ (by considering the fibers of the map $E\to \olE:\{v,w\}\mapsto \{[v],[w]\}$). The next lemma shows us, among other, that the map
\begin{equation}
    \label{eq:identificationEbarAndM}
    M \to \overline{E}: [\{v,w\}] \mapsto \{ [v], [w] \}
\end{equation}
is a bijection. By consequence, we can identify $M$ with $\overline{E}$.
\begin{lemma}\label{lem:equivalence class e}
    For any $e=\{v,w\}\in E$ it holds that
    \[[e]= \Big\{\{v',w'\} \: \Big| \: v'\in [v], w'\in [w] \Big\}.\]
\end{lemma}
\begin{proof}
The $\supseteq$-inclusion follows directly by using the definition of $\prec$ and $\sim$ on $E$. For the other inclusion let us fix some $e'=\{v',w'\}\in [e]$. Hence, it holds that
\begin{align*}
    (v'\prec v \text{ and } w'\prec w) &\text{ or } (v'\prec w \text{ and } w'\prec v) \text{ and }\\
    (v\prec v' \text{ and } w\prec w') &\text{ or } (v\prec w' \text{ and } w\prec v').
\end{align*}
Note that if for example ($v'\prec v \text{ and } w'\prec w$) and ($v\prec w' \text{ and } w\prec v'$) then $v\prec w'\prec w\prec v' \prec v$ and thus it follows that $v\sim w\sim v'\sim w'$. Hence, in this particular case $e'$ is definitely contained in the right-hand side. Considering all the possible cases, one can conclude the rest of the proof.
\end{proof}

Moreover, the preorders $\prec$ on the vertices and the edges induce partial orders $\olprec$ on $\Lambda$ and $M$ by setting for all $v,w\in V$ and $e,e'\in E$ that
\begin{align*}
    [v]\olprec[w] &\quad\Longleftrightarrow\quad v\prec w\\
    [e]\olprec[e'] &\quad\Longleftrightarrow\quad e\prec e'.
\end{align*}
Using these partial orders, we can fix total orders
\[V/\sim\, = \Lambda = \{\lambda_1,\dots,\lambda_r\} \quad \text{and} \quad E/\sim\, = M = \{\mu_1,\dots,\mu_s\}\] with the property that $i\leq j$ if $\lambda_i\olprec\lambda_j$ and $i\leq j$ if $\mu_i\olprec\mu_j$. We can refine these orders to total orders on $V$ and $E$ which we will both denote with $<$.

\begin{example}
	We explain the introduced notation using a concrete example. We consider the graph $\Gamma$ in Figure \ref{subfig:Graph}.
	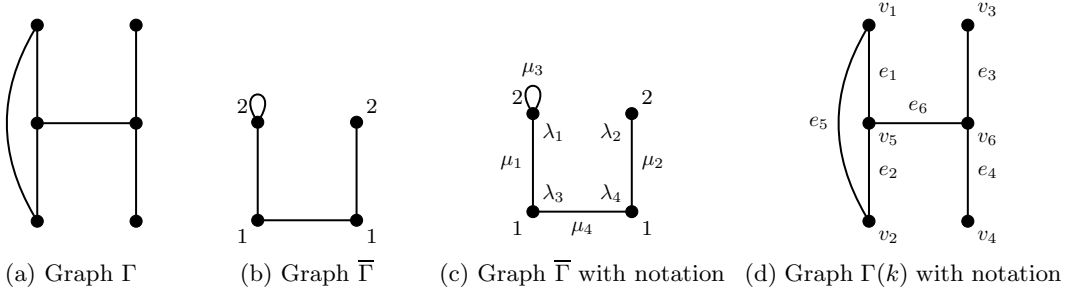
\begin{figure}[H]
		\begin{subfigure}{0.2\textwidth}
			\centering
			\begin{tikzpicture}[main_node/.style={circle,draw,fill,minimum size=4.5pt,inner sep=0, outer sep=0]}]
				\def\lengte{1.3cm}
				\node[main_node] (0) at (0,{2*\lengte}) {};
				\node[main_node] (1) at (0,\lengte) {};
				\node[main_node] (2) at (0,0) {};
				\node[main_node] (3) at (\lengte,{2*\lengte}) {};
				\node[main_node] (4) at (\lengte,\lengte) {};
				\node[main_node] (5) at (\lengte,0) {};
				\node[anchor=north east,xshift=-2pt,yshift=-2pt,outer sep=0, inner sep=0] at (5.south west) {\phantom{\footnotesize{1}}};
				
				\path[draw, line width=0.75pt]
				(0) edge 	node {} (1) 
				(1) edge 	node {} (2)
				(0) edge[bend right] 	node {} (2) 
				(1) edge 	node {} (4) 
				(3) edge 	node {} (4) 
				(4) edge 	node {} (5)  
				;
			\end{tikzpicture}
			\caption{Graph $\Gamma$}
			\label{subfig:Graph}
		\end{subfigure}
		\begin{subfigure}{0.2\textwidth}
			\centering
			\begin{tikzpicture}[main_node/.style={circle,draw,fill,minimum size=4.5pt,inner sep=0, outer sep=0]}]
				\def\lengte{1.3cm}
				\node[main_node] (0) at (0,\lengte) {};
				\node[anchor=south east,xshift=-2pt,yshift=2pt,outer sep=0, inner sep=0] at (0.north west) {\footnotesize{2}};
				\node[main_node] (1) at (0,0) {};
				\node[anchor=north east,xshift=-2pt,yshift=-2pt,outer sep=0, inner sep=0] at (1.south west) {\footnotesize{1}};
				\node[main_node] (2) at (\lengte,\lengte) {};
				\node[anchor=south west,xshift=2pt,yshift=2pt,outer sep=0, inner sep=0] at (2.north east) {\footnotesize{2}};
				\node[main_node] (3) at (\lengte,0) {};
				\node[anchor=north west,xshift=2pt,yshift=-2pt,outer sep=0, inner sep=0] at (3.south east) {\footnotesize{1}};
				
				\path[draw, line width=0.75pt]
				(0) edge[out=60,in=120,looseness=15] node {} (0)
				(0) edge node {} (1) 
				(1) edge node {} (3) 
				(3) edge node {} (2) 
				;
			\end{tikzpicture}
			\caption{Graph $\olGamma$}
			\label{subfig:QuotientGraph}
		\end{subfigure}
		\begin{subfigure}{0.27\textwidth}
			\centering
			\begin{tikzpicture}[main_node/.style={circle,draw,fill,minimum size=4.5pt,inner sep=0, outer sep=0]}]
				\def\lengte{1.3cm}
				\node[main_node] (0) at (0,\lengte) {};
				\node[anchor=south east,xshift=-2pt,yshift=2pt,outer sep=0, inner sep=0] at (0.north west) {\footnotesize{2}};
				\node[anchor=north west,xshift=2pt,yshift=-2pt,outer sep=0, inner sep=0] at (0.south east) {\scalebox{0.8}{$\lambda_1$}};
				\node[main_node] (1) at (0,0) {};
				\node[anchor=north east,xshift=-2pt,yshift=-2pt,outer sep=0, inner sep=0] at (1.south west) {\footnotesize{1}};
				\node[anchor=south west,xshift=2pt,yshift=2pt,outer sep=0, inner sep=0] at (1.north east) {\scalebox{0.8}{$\lambda_3$}};
				\node[main_node] (2) at (\lengte,\lengte) {};
				\node[anchor=south west,xshift=2pt,yshift=2pt,outer sep=0, inner sep=0] at (2.north east) {\footnotesize{2}};
				\node[anchor=north east,xshift=-2pt,yshift=-2pt,outer sep=0, inner sep=0] at (2.south west) {\scalebox{0.8}{$\lambda_2$}};
				\node[main_node] (3) at (\lengte,0) {};
				\node[anchor=north west,xshift=2pt,yshift=-2pt,outer sep=0, inner sep=0] at (3.south east) {\footnotesize{1}};
				\node[anchor=south east,xshift=-2pt,yshift=2pt,outer sep=0, inner sep=0] at (3.north west) {\scalebox{0.8}{$\lambda_4$}};
				
				\path[draw, line width=0.75pt]
				(0) edge[out=60,in=120,looseness=15] node[above] {\scalebox{0.8}{$\mu_3$}} (0)
				(0) edge node[left] {\scalebox{0.8}{$\mu_1$}} (1) 
				(1) edge node[below] {\scalebox{0.8}{$\mu_4$}} (3) 
				(3) edge node[right] {\scalebox{0.8}{$\mu_2$}} (2) 
				;
			\end{tikzpicture}
			\caption{Graph $\olGamma$ with notation}
			\label{subfig:QuotientGraphWithNotation}
		\end{subfigure}
		\begin{subfigure}{0.28\textwidth}
			\centering
			\begin{tikzpicture}[main_node/.style={circle,draw,fill,minimum size=4.5pt,inner sep=0, outer sep=0]}]
				\def\lengte{1.3cm}
				\node[main_node] (0) at (0,{2*\lengte}) {};
				\node[anchor=south west,xshift=2pt,yshift=2pt,outer sep=0, inner sep=0] at (0.north east) {\scalebox{0.8}{$v_1$}};
				\node[main_node] (1) at (0,\lengte) {};
				\node[anchor=north west,xshift=2pt,yshift=-2pt,outer sep=0, inner sep=0] at (1.south east) {\scalebox{0.8}{$v_5$}};
				\node[main_node] (2) at (0,0) {};
				\node[anchor=north west,xshift=2pt,yshift=-2pt,outer sep=0, inner sep=0] at (2.south east) {\scalebox{0.8}{$v_2$}};
				\node[main_node] (3) at (\lengte,{2*\lengte}) {};
				\node[anchor=south west,xshift=2pt,yshift=2pt,outer sep=0, inner sep=0] at (3.north east) {\scalebox{0.8}{$v_3$}};
				\node[main_node] (4) at (\lengte,\lengte) {};
				\node[anchor=north west,xshift=2pt,yshift=-2pt,outer sep=0, inner sep=0] at (4.south east) {\scalebox{0.8}{$v_6$}};
				\node[main_node] (5) at (\lengte,0) {};
				\node[anchor=north west,xshift=2pt,yshift=-2pt,outer sep=0, inner sep=0] at (5.south east) {\scalebox{0.8}{$v_4$}};
				\node[anchor=north east,xshift=-2pt,yshift=-2pt,outer sep=0, inner sep=0] at (5.south west) {\phantom{\footnotesize{1}}};
				
				\path[draw, line width=0.75pt]
				(0) edge 	node[right] {\scalebox{0.8}{$e_1$}} (1) 
				(1) edge 	node[right] {\scalebox{0.8}{$e_2$}} (2)
				(0) edge[bend right] 	node[left] {\scalebox{0.8}{$e_5$}} (2) 
				(1) edge 	node[above] {\scalebox{0.8}{$e_6$}} (4) 
				(3) edge 	node[right] {\scalebox{0.8}{$e_3$}} (4) 
				(4) edge 	node[right] {\scalebox{0.8}{$e_4$}} (5)
				;
			\end{tikzpicture}
			\caption{Graph $\Gamma(k)$ with notation}
			\label{subfig:GraphWithNotation}
		\end{subfigure}
		\caption{Graphs $\Gamma(k)$ and $\olGamma$ with introduced notation}
		\label{fig:ExampleNotation}
	\end{figure}
	We construct the quotient graph $\olGamma$ and we display the size of each coherent component next to the corresponding vertex of $\olGamma$ (see Figure \ref{subfig:QuotientGraph}). Next, we look at the induced relation $\olprec$ on $\Lambda$. We could for example fix the total order on $\Lambda$ as depicted in Figure \ref{subfig:QuotientGraphWithNotation}. Indeed, the relations on $\Lambda$ are
    \[\lambda_1\olprec \lambda_3,\: \lambda_2\olprec \lambda_3 \text{ and } \lambda_2\olprec \lambda_4\]
    and thus this total order satisfies that $i\leq j$ if $\lambda_i\olprec\lambda_j$. We could have for example also interchanged $\lambda_1$ and $\lambda_2$. Now we fix an order inside each coherent component and get the total order on $V$ depicted in Figure \ref{subfig:GraphWithNotation}.

    The only relations on the edges are
    \[\{v_1,v_5\}\sim \{v_2,v_5\} \text{ and } \{v_3,v_6\}\sim \{v_4,v_6\}\]
    and thus the induced partial order $\olprec$ on $\olE$ is the identity relation. Hence, we could for example fix the total order
    \[[\{v_1,v_5\}]<[\{v_3,v_6\}]<[\{v_1,v_2\}]<[\{v_5,v_6\}]\]
    on $\olE$ and refine the order to the following total order on $E$
    \[\{v_1,v_5\}<\{v_2,v_5\}<\{v_3,v_6\}<\{v_4,v_6\}<\{v_1,v_2\}<\{v_5,v_6\}.\]
\end{example}

\subsection{Definition and structure of the groups \texorpdfstring{$G_{\Gamma(k)}$}{G_Gamma(k)}}\label{sec:definition Ggamma(k)}
The total order on the vertices and edges allows us to construct and study multiple groups associated to a graph $\Gamma$. For a graph $\Gamma(k)$ with weight function $k:E\to \N_0$, we define $G_{\Gamma(k)}$ by
\begin{align*}
	G_{\Gamma(k)}&:=\Bigg\langle V \cup E \;\; \Bigg| \;\; [V,E]= [E, E] = 1 \text{ and } \forall v, w \in V\text{ with }v<w:\\
	& \hspace{15em} [v,w]=\begin{cases}
		e^{k(e)} &\text{if }  e:=\{v,w\}\in E\\
		1 &\text{else}
	\end{cases}	\Bigg\rangle.
\end{align*}
Note that $G_{\Gamma(k)}\cong G_{\Gamma}$ if $k:E\to\N_0: e\mapsto 1$.

\begin{remark}
    The definition of $G_{\Gamma(k)}$ is independent of the choice of total order on the vertices in the sense that one gets isomorphic groups for different total orders. Indeed, let $\tilde{<}$ be another total order on the vertices $V$. Then it is not hard to check that the map $\theta$, which is the identity on $V$ and is defined for all $e=\{v,w\}\in E$ by
    \[\theta(e)=\begin{cases}
        e &\text{if } v<w \Leftrightarrow v\tilde{<}w\\
        e^{-1} &\text{otherwise},
    \end{cases}\]
    induces an isomorphism from the group defined using $<$ to the group defined using $\tilde{<}$.
\end{remark}

\begin{example}
	Consider the graph $\Gamma = (V, E)$ where $V = \{v_1, v_2\}$ and $E = \{\{v_1, v_2\}\}$ and for any $n \in \N_0$, let $k:E \to \N_0$ be the weight function on the singleton set $E$ with value $n$.
    \begin{figure}[H]
        \centering
			\begin{tikzpicture}[main_node/.style={circle,draw,fill,minimum size=4.5pt,inner sep=0, outer sep=0]}]
				\def\lengte{1.3cm}
				\node[main_node] (0) at (0,0) {};
                \node[anchor=south east,xshift=-2pt,yshift=2pt,outer sep=0, inner sep=0] at (0.north west) {$v_1$};
				\node[main_node] (1) at (\lengte,0) {};
                \node[anchor=south west,xshift=2pt,yshift=2pt,outer sep=0, inner sep=0] at (1.north east) {$v_2$};
				
				\path[draw, line width=0.75pt]
				(0) edge 	node[above] {$n$} (1)  
				;
			\end{tikzpicture}
    \end{figure}
    Then for each $n \in \N_0$, the group $G_{\Gamma(k)}$ is isomorphic to the matrix group
	\[ H_n = \left\{ \begin{pmatrix}
		1 & nx & z\\
		0 & 1 & y\\
		0 & 0 & 1
	\end{pmatrix} \: \middle| \: x,y,z \in \Z \right\} \]
	where the isomorphism is given by
	\[ v_1 \mapsto \begin{pmatrix}
		1 & n & 0\\
		0 & 1 & 0\\
		0 & 0 & 1
	\end{pmatrix}, \quad v_2 \mapsto \begin{pmatrix}
		1 & 0 & 0\\
		0 & 1 & 1\\
		0 & 0 & 1
	\end{pmatrix}, \quad \{v_1, v_2\} \mapsto \begin{pmatrix}
		1 & 0 & 1\\
		0 & 1 & 0\\
		0 & 0 & 1
	\end{pmatrix}\]
	In particular, when $n = 1$, $G_{\Gamma(k)} \cong G_{\Gamma}$ is isomorphic to the Heisenberg group. It is well-known that the groups $H_n$ for $n \in \N_0$ are, up to isomorphism, all the possible groups abstractly commensurable to $H_1$. Moreover, $\Spec(H_n)=2\N_0\cup \{\infty\}$ for all $n\in \N_0$ (see for example \cite{roma11-1}). Hence, $H_1$ is an example of a finitely generated torsion-free 2-step nilpotent group such that all groups commensurable with $H_1$ have the same (infinite) Reidemeister spectrum.
\end{example}

It will be helpful to address the elements $e^{k(e)} \in G_{\Gamma(k)}$, so we define
\[ E^k:=\{e^{k(e)}\mid e\in E\}. \]
Recall that $V_{\text{iso}} \subset V$ denotes the set of \textit{isolated vertices} of $\Gamma$.

The next lemma provides insight into the structure of the group $G_{\Gamma(k)}$.
\begin{lemma}\label{lem:structure GGammak}
        Denote $V=\{v_1,\dots,v_n\}$ and $E=\{e_1,\dots,e_m\}$. Then any element of $G_{\Gamma(k)}$ can be uniquely expressed as a product
	\[\prod_{i=1}^{n}v_i^{z_i}\prod_{j=1}e^{t_j}\]
	with $z_i,t_j\in \Z$. Moreover, the centre and the factors of the (adapted) lower central series of $G_{\Gamma(k)}$ are given by
	\begin{alignat*}{2}
        Z(G_{\Gamma(k)})&= \left\langle V_{\text{iso}} \cup E \right\rangle \cong \mathbb{Z}^{\vert V_{\text{iso}} \cup E \vert}, & &\\
	\gamma_2(G_{\Gamma(k)})&=\left\langle E^k\right\rangle\cong \Z^{\vert E\vert}, \quad \quad& \frac{G_{\Gamma(k)}}{\gamma_2(G_{\Gamma(k)})}&\cong \Z^{\vert V\vert}\times \bigtimes_{e\in E} \frac{\Z}{k(e)\Z},\\
	\sqrt{\gamma_2(G_{\Gamma(k)})}&=\langle E\rangle\cong \Z^{\vert E\vert}, \quad \quad& \frac{G_{\Gamma(k)}}{\sqrt{\gamma_2(G_{\Gamma(k)})}}&=\bigtimes_{v\in V}\left\langle v\sqrt{\gamma_2(G_{\Gamma(k)})}\right\rangle\cong \Z^{\vert V\vert}.
	\end{alignat*}
	Moreover, the identity on $V$ induces an injective morphism $\iota:G_{\Gamma}\xhookrightarrow{} G_{\Gamma(k)}$ with $G_{\Gamma}\cong \Im(\iota)=\langle V\cup E^k\rangle$.
	In particular, $G_{\Gamma}$ is isomorphic to the finite index subgroup $\langle V\cup E^k\rangle\subset G_{\Gamma(k)}$ with index $\prod_{e\in E} k(e)$.
\end{lemma}

Note that $\gamma_2(G_{\Gamma(k)})\subset Z(G_{\Gamma(k)})$ and thus $G_{\Gamma(k)}$ is a (finitely generated torsion-free) 2-step nilpotent group. In section \ref{sec:LieAlgAssToGrps} we associated two Lie algebras to these kinds of groups. As it turns out, in this case, they are isomorphic and do not depend on the chosen weight function.



\begin{lemma}\label{lem:AllLieAlgebrasIso}
	For any finite undirected simple graph $\Gamma=(V,E)$ and a weight function $k:E\to \N_0$ on its edges, the Lie algebras $L^{\Q}(\Gamma,2)$, $L^{\Q}(G_{\Gamma})$, $L^{\Q}(G_{\Gamma(k)})$, $M(G_\Gamma)$ and $M(G_{\Gamma(k)})$ are all isomorphic and the isomorphisms are induced by the identity on the vertices.
\end{lemma}

\begin{proof}
    Note that any 2-step nilpotent rational Lie algebra $L$ is isomorphic to $\gr(L)$ and that the isomorphism is canonical up to a choice of vector space complement to $[L, L]$. Since for $M(G_\Gamma)$ and $M(G_{\Gamma(k)})$, such a complement is naturally given by the span of the vertices, one gets (after combining with Equation (\ref{eq:isoMalcevGraded})) the isomorphisms $M(G_\Gamma) \cong \gr(M(G_\Gamma)) \cong L^\Q(G_\Gamma)$ and $M(G_{\Gamma(k)}) \cong \gr(M(G_{\Gamma(k)})) \cong L^\Q(G_{\Gamma(k)})$. As one can check, all these isomorphisms restrict to the identity on the vertices.

    Next, since the group $G_\Gamma$ is a finite index subgroup of $G_{\Gamma(k)}$, it follows (see section \ref{prop:isoGradLieAlgImplicationRinfty}) that the associated Mal'cev Lie algebras are isomorphic and that the isomorphism restricts to the inclusion of $G_\Gamma$ in $G_{\Gamma(k)}$ and thus to the identity on the vertices.

    At last, using Equation (\ref{eq:isomorphismRAAGLieAlgGrp}), we find that $L^\Q(\Gamma, 2) \cong L^\Q(G_\Gamma)$ where the isomorphism restricts to the identity on the vertices. This completes the proof.
\end{proof}

\newpage
\section{Automorphisms of \texorpdfstring{$G_{\Gamma(k)}$}{G_Gamma(k)}}
\label{sec:automorpismsOfWeightedRAAG}

In this section, we describe (to some extent) the automorphisms of $G_{\Gamma(k)}$. In the first part (section \ref{sec:AutLieAlg}), we recall a result from \cite{dm21-1}, where the automorphisms of $L^K(\Gamma, 2)$ are described for any field $K\subset \C$. Since $L^\Q(\Gamma, 2) \cong L^\Q(G_{\Gamma(k)})$ for any weight function $k$ (see Lemma \ref{lem:AllLieAlgebrasIso}), this result gives us information about the automorphisms of $G_{\Gamma(k)}$ (recall that any automorphism of $G_{\Gamma(k)}$ induces an automorphism on $L^\Q(G_{\Gamma(k)})$).
In the second part (section \ref{sec:AutGrp}), we argue on the group level to get more restrictions on the possible automorphisms. These restrictions are imposed by the weight function $k:E \to \N_0$.

\subsection{Automorphisms of \texorpdfstring{$L^K(\Gamma, 2)$}{L^K(Gamma,2)}}
\label{sec:AutLieAlg}
In order to describe the automorphisms of $L^K(\Gamma, 2)$, one needs some extra maps and notation:

Recall the definition from the quotient graph $\olGamma=(\Lambda,\olE,\Psi)$ of $\Gamma=(V,E)$. We say that $\sigma\in \Aut(\olGamma)$ is an \textit{automorphism of the quotient graph} if it is a bijection $\sigma:\Lambda\to \Lambda$ such that edges are mapped to edges (i.e. $\sigma(\olE)=\olE$) and the sizes of the coherent components are invariant (i.e. $\Psi\circ \sigma=\Psi$). Note that any graph automorphism $\sigma\in \Aut(\Gamma)$ induces a unique automorphism $p(\sigma)\in \Aut(\olGamma)$ by setting $p(\sigma)([v]):=[\sigma(v)]$ for any vertex $v\in V$. This gives a surjective group homomorphism $p: \Aut(\Gamma) \to \Aut(\olGamma)$.
    

For any basis $X$ of a vector space $W$ over $K$ and any permutation $\sigma$ on the basis $X$, we write $P(\sigma)$ for the linear map on $W$ (or for the corresponding matrix) defined by $P(\sigma)(x):=\sigma(x)$ (for all $x\in X$).

We define the linear endomorphisms $E_{vw}$ (for all vertices $v,w\in V$) on $\Sp_K(V)$ (for any $v'\in V$) by
\[ E_{vw}(v')=\begin{cases}
	v & \text{if } v'=w\\
	0 & \text{if } v'\neq w
\end{cases}. \]

Denote with $I\in \Gl(\Sp_K(V))$ the identity map and with $\mathcal{U}\subseteq \Gl(\Sp_K(V))$ the subgroup defined by
\[ \mathcal{U}:= \langle I+t E_{vw}\:\mid\:v\prec w,\: w\not\prec v,\: t\in K\rangle.\]

We use $\Gl(\Sp_K(\lambda))$ (for any $\lambda\in \Lambda$) to denote the subgroup of $\Gl(\Sp_K(V))$ where $f\in\Gl(\Sp_K(\lambda))$ induces a linear map (of $\Gl(\Sp_K(V))$) by setting (for any $v\in V$)
\[ v\mapsto\begin{cases}
	f(v) & \text{if } v\in \lambda\\
	v & \text{if } v\not\in \lambda
\end{cases}. \]

We have now introduced enough notation to describe the graded automorphisms of $L^K(\Gamma, 2)$ (with $K\subset \C$). Recall Remark \ref{rem:RestrictionToAbelianization}.
\begin{theorem}[\cite{dm21-1}]\label{thm:GradedAutomorphismsLieAlgebra}
	Let $K\subset \C$ be a field. Then the image of $\pi:\Aut_g(L^K(\Gamma, 2))\to \Gl(\Sp_K(V))$ is a linear algebraic group given by
	\[ \mathcal{G}_\Gamma := P(\Aut(\Gamma))\cdot \biggg(\prod_{\lambda\in \Lambda}\Gl(\Sp_K(\lambda))\biggg)\cdot \mathcal{U} \]
	where $\mathcal{U}$ is its unipotent radical. In particular, the matrix of any element $f\in \mathcal{G}_\Gamma$ with respect to the basis $(V,<)$ has the form (where the blocks correspond with the coherent components $\lambda_1<\dots<\lambda_r$)
	\[ P(\sigma)\cdot\begin{pmatrix}
		A_1 & A_{12} & \hdots & A_{1r} \\
		0 & A_2 & \ddots & \vdots \\
		\vdots & \ddots & \ddots & A_{r-1 r} \\
		0 & \hdots & 0 & A_r
	\end{pmatrix} \]
	where $\sigma\in \Aut(\Gamma)$, $A_i\in \Gl_{\vert \lambda_i\vert }(K)$, $A_{ij}\in K^{\vert \lambda_i\vert \times \vert \lambda_j\vert }$ and $A_{ij}=0$ if $\lambda_i\not\olprec\lambda_j$. The permutation $\sigma$ is not necessarily unique, but the induced permutation $p(\sigma) \in \Aut(\olGamma)$ is unique. Moreover, the assignment $f \mapsto p(\sigma)$ is a group homomorphism from $\mathcal{G}_\Gamma$ to $\Aut(\overline{\Gamma})$.
\end{theorem}
Recall that any automorphism $\varphi\in \Aut(G_{\Gamma(k)})$ induces an automorphism $\olvarphi\in \Aut_g(L^\Q(G_{\Gamma(k)})) \cong \Aut_g(L^\Q(\Gamma, 2))$. Hence, by Theorem \ref{thm:GradedAutomorphismsLieAlgebra} it induces a (unique) automorphism $p(\sigma) \in \Aut(\olGamma)$. In the next section, we study which automorphisms of $\olGamma$ can arise in this way. Later, we use this (together with Lemma \ref{lem:eigenvalue1}) to derive some restrictions on an edge-weighted graph (with weighted edges) such that the associated group has the $R_{\infty}$--property.

\subsection{Automorphisms of \texorpdfstring{$G_{\Gamma(k)}$}{G_Gamma(k)}}
\label{sec:AutGrp}
\begin{notation}
	Let $(B=\{b_1,\dots,b_n\},<)$ be a finite totally ordered set (with $b_1<\dots<b_n$) and $\theta:B\to \Z$ a map. Then we define $D(\theta)\in \Z^{n\times n}$ to be the diagonal matrix with the image of $\theta$ (ordered by $<$ on $B$) on its diagonal, i.e.
	\[ D(\theta):=\text{diag}(\theta(b_1),\dots,\theta(b_n))=\begin{pmatrix}
		\theta(b_1) & 0 & \hdots & 0 \\
		0 & \theta(b_2) & \ddots & \vdots \\
		\vdots & \ddots & \ddots & 0 \\
		0 & \hdots & 0 & \theta(b_n)
	\end{pmatrix}.\]
	If $B$ is a basis of a $\Z$-module $M$, then we will sometimes abuse notation and denote with $D(\theta)$ also the linear map on $M$ induced by $D(\theta)(b_i):=\theta(b_i)b_i$ (for all $i=1,\dots,n$).
\end{notation}
\begin{lemma}\label{lem:AutomorphismOnSqrtGamma2}
	Fix a finite undirected simple graph $\Gamma=(V,E)$ with weights $k:E\to \N_0$ on its edges. Let $\varphi\in \Aut(G_{\Gamma(k)})$ and denote with $A\in \Gl_{\vert E\vert }(\Z)$ the matrix of $\varphi\vert _{\gamma_2(G_{\Gamma(k)})}$ with respect to the $\Z$-basis $(E^k,<)$.\\
	Then the matrix of $\varphi\vert _{\sqrt{\gamma_2(G_{\Gamma(k)})}}$ with respect to $(E,<)$ is given by $D(k)AD(k)^{-1}$. In particular, $D(k)AD(k)^{-1}\in \Gl_{\vert E\vert }(\Z)$.
\end{lemma}
\begin{proof}
	Since the induced automorphisms on $\gamma_2(G_{\Gamma(k)})$ and $\sqrt{\gamma_2(G_{\Gamma(k)})}$ are restrictions of $\varphi$, we obtain the following commuting diagram.
	\[ \begin{tikzcd}
		\gamma_2(G_{\Gamma(k)}) \arrow[r, "\varphi\vert_{\gamma_2}"] \arrow[d, hook, "\iota"] & \gamma_2(G_{\Gamma(k)}) \arrow[mysymbol]{dl}[description]{\huge{\circlearrowleft}} \arrow[d, hook, "\iota"] \\
		\sqrt{\gamma_2(G_{\Gamma(k)})} \arrow[r, "\varphi\vert_{\sqrt{\gamma_2}}"']  & \sqrt{\gamma_2(G_{\Gamma(k)})}
	\end{tikzcd} \]
	Recall that $\gamma_2(G_{\Gamma(k)})$ and $\sqrt{\gamma_2(G_{\Gamma(k)})}$ are both free abelian of rank $\vert E\vert $. By Lemma \ref{lem:structure GGammak}, the sets $E$ and $E^k = \{e^{k(e)}\mid e\in E\}$ ordered with the total order $<$ give a basis for the $\Z$-modules $\sqrt{\gamma_2(G_{\Gamma(k)})}$ and $\gamma_2(G_{\Gamma(k)})$, respectively. Hence, by using these bases we get the commuting diagram of matrices
	\[ \begin{tikzcd}
		\Z^{\vert  E\vert  } \arrow[r, "A"] \arrow[d, "D"'] & \Z^{\vert  E\vert } \arrow[mysymbol]{dl}[description]{\huge{\circlearrowleft}} \arrow[d, "D"] \\
		\Z^{\vert  E\vert } \arrow[r, "B"']  & \Z^{\vert  E\vert }
	\end{tikzcd} \]
	where $A, B\in \Gl_{\vert E\vert }(\Z)$ are the matrix representations of $\varphi\vert_{\gamma_2(G_{\Gamma(k)})}$ and $\varphi|_{\sqrt{\gamma_2(G_{\Gamma(k)})}}$, respectively, and $D\in \Z^{\vert E\vert \times \vert E\vert}$ is the matrix representation of the inclusion $\iota$. Note that $D=D(k)$ since a basis vector $e^{k(e)}$ of $\gamma_2(G_{\Gamma(k)})$ is sent to the basis vector $e$ of $\sqrt{\gamma_2(G_{\Gamma(k)})}$ but raised to the power $k(e)$. Thus the commutative diagram implies that $BD(k)=D(k)A$. Since $k(e)\neq 0$ (for all $e\in E$), these matrices are invertible over $\Q$. Over $\Q$, we indeed obtain that $B=D(k)AD(k)^{-1}$.
\end{proof}
\begin{remark}
	Recall that any automorphism of $G_{\Gamma(k)}$ is completely determined by the image of the vertices. Indeed, by Lemma \ref{lem:AutomorphismOnSqrtGamma2} the image of $E$ is fixed by what happens on $\gamma_2(G_{\Gamma(k)})$ which on its turn is determined by the image of the vertices. Since $V$ and $E$ generate $G_{\Gamma(k)}$ the claim follows. Which matrices $A\in \Gl_{\vert V\vert }(\Z)$ on the quotient of $G_{\Gamma(k)}$ with $\sqrt{\gamma_2(G_{\Gamma(k)})}$ induce an automorphism on $G_{\Gamma(k)}$? One can prove that it suffices to require that it induces an automorphism of $G_{\Gamma}$ and that the induced matrix on $\sqrt{\gamma_2(G_{\Gamma(k)})}$ (which we described in Lemma \ref{lem:AutomorphismOnSqrtGamma2}) has integer entries.\\
    Note however that any such matrix $A\in \Gl_{\vert V\vert }(\Z)$ does not induce a unique automorphism of $G_{\Gamma(k)}$. Indeed, one can always multiply the image of a vertex with any element of $\sqrt{\gamma_2(G_{\Gamma(k)})}$. Nevertheless, the induced automorphisms on the Lie ring $L(G_{\Gamma(k)})$ are the same and they have the same Reidemeister number. Hence, since we are interested in the Reidemeister numbers of automorphisms of $G_{\Gamma(k)}$, we can assume without loss of generality that $\varphi(\langle V\rangle)=\langle V\rangle$ for the automorphisms we work with.
\end{remark}

Using the description of the automorphisms of the Lie algebra $L^\Q(\Gamma, 2)$ of above, we can describe to some extent the automorphisms of $G_{\Gamma(k)}$ and in particular their induced maps on $\gamma_2(G_{\Gamma(k)})$. In order to do so, we need the following notation.

Let $\Gamma = (V, E)$ be a graph with an automorphism $\sigma \in \Aut(\Gamma)$. Then $\sigma$ induces a permutation $\sigma_E$ on the edge set by
\[ \sigma_E:E \to E: \{v,w\} \mapsto \{\sigma(v), \sigma(w)\}.\]
If there is a total order $<$ on the vertices $V$, we say that an edge $\{v, w\} \in E$ is an \textit{inversion of $\sigma$} if $v<w$, but $\sigma(v) > \sigma(w)$. We then define the map $\varepsilon_\sigma:E\to \{-1,1\}$ by setting for all $e\in E$
	\[ \varepsilon_\sigma(e):=\begin{cases}
		-1 &\text{if } e \text{ is an inversion of } \sigma\\
		1 &\text{otherwise}.
	\end{cases} \]
With this notation at hand, we can say the following about the automorphisms of $G_{\Gamma(k)}$. Recall from section \ref{sec:relationsOnVerticesAndEdges} that we defined a total order $<$ on both the vertices and the edges, which were a refinement of the orders on $\Lambda = \{\lambda_1,\dots,\lambda_r\}$ and $M = \{\mu_1,\dots,\mu_s\}$, respectively.
\begin{lemma}
    \label{lem:inducedAutoAbAndGamma2}
    Let $\Gamma=(V,E)$ be a finite undirected simple graph with weight function $k:E\to \N_0$ on its edges and write $G = G_{\Gamma(k)}$. For any automorphism $\varphi\in \Aut(G)$ it holds that
    \begin{enumerate}[label = (\roman*)]
        \item with respect to the ordered $\Z$-basis $(V, <)$, the induced automorphism on $G/\sqrt{\gamma_2(G)}$ is represented by a matrix of the form
        \[ P(\sigma)\cdot
        \begin{pmatrix}
		      A_1 & A_{12} & \hdots & A_{1r} \\
		      0 & A_2 & \ddots & \vdots \\
		      \vdots & \ddots & \ddots & A_{r-1 r} \\
		      0 & \hdots & 0 & A_r
	    \end{pmatrix} \]
        where $\sigma \in \Aut(\Gamma)$, $A_i\in \Gl_{\vert \lambda_i\vert }(\Z)$, $A_{ij}\in \Z^{\vert \lambda_i\vert \times \vert \lambda_j\vert }$ and $A_{ij}=0$ if $\lambda_i\not\olprec\lambda_j$.
        \item with respect to the ordered $\Z$-basis $(E^k, <)$, the induced automorphism on $\gamma_2(G)$ is represented by a matrix of the form
        \[ P(\sigma_E) \cdot D(\varepsilon_\sigma) \cdot
        \begin{pmatrix}
		      B_1 & B_{12} & \hdots & B_{1s} \\
		      0 & B_2 & \ddots & \vdots \\
		      \vdots & \ddots & \ddots & B_{s-1 s} \\
		      0 & \hdots & 0 & B_s
	    \end{pmatrix} \]
        where $B_i \in \Gl_{|\mu_i|}(\Z)$, $B_{ij} \in \Z^{|\mu_i| \times |\mu_j|}$ and $B_{ij}=0$ if $\mu_i\not\olprec\mu_j$.
    \end{enumerate}
    The permutation $\sigma \in \Aut(\Gamma)$ is not necessarily unique, but the induced permutation $p(\sigma) \in \Aut(\overline{\Gamma})$ is unique. Moreover, the assignment $\varphi \mapsto p(\sigma)$ defines a group homomorphism from $\Aut(G)$ to $\Aut(\overline{\Gamma})$.
\end{lemma}

\begin{proof}
    Fix $\varphi\in \Aut(G)$ and denote with $\olvarphi\in \Aut_g(L^\Q(\Gamma,2))$ the induced graded automorphism. By Theorem \ref{thm:GradedAutomorphismsLieAlgebra} there exist $A_i\in \Gl_{\vert \lambda_i\vert}(\Q)$, $A_{ij}\in \Q^{\vert \lambda_i\vert \times \vert \lambda_j\vert}$ and $\sigma\in \Aut(\Gamma)$ (with $p(\sigma)\in \Aut(\olGamma)$ unique) such that the matrix of $\olvarphi\vert_{\Sp_\Q(V)}$ with respect to $(V,<)$ equals
	\[ P(\sigma)\cdot\underbrace{\begin{pmatrix}
		A_1 & A_{12} & \hdots & A_{1r} \\
		0 & A_2 & \ddots & \vdots \\
		\vdots & \ddots & \ddots & A_{r-1 r} \\
		0 & \hdots & 0 & A_r
	\end{pmatrix}}_{=:U}\]
    where $A_{ij}=0$ if $\lambda_i\not\olprec\lambda_j$. Since $\olvarphi$ is induced by the automorphism $\varphi\in \Aut(G)$, the matrix $P(\sigma)U$ is by construction equal to the matrix of the induced automorphism on $G/\sqrt{\gamma_2(G)}$ with respect to the $\Z$-basis $(V/\sqrt{\gamma_2(G_{\Gamma(k)})},<)$ (where the order is induced by $<$). In particular, $P(\sigma)U\in \Gl_{\vert V\vert}(\Z)$. Since $P(\sigma)\in \Gl_{\vert V\vert}(\Z)$, this implies that $U\in \Gl_{\vert V\vert}(\Z)$. Hence, we obtain that $A_i\in \Gl_{\vert \lambda_i\vert}(\Z)$ and $A_{ij}\in \Z^{\vert \lambda_i\vert \times \vert \lambda_j\vert}$.\\
    Now we derive the matrix of the induced automorphism on $\gamma_2(G)$ with respect to the ordered $\Z$-basis $(E^k,<)$.	Fix any edge $e=\{v,w\}\in E$ (with $v<w$) and note that the Lie bracket $[v,w]=e^{k(e)}$ is sent by the Lie algebra morphism $P(\sigma)$ to
	\[ [\sigma(v),\sigma(w)]=\sigma_E(e)^{\varepsilon_\sigma(e)k(\sigma_E(e))}.\]
	Thus the matrix (with respect to $(E^k,<)$) of the morphism on $\gamma_2(L^\Q(\Gamma,2))$ that is induced by $P(\sigma)$ equals
	\[ P(\sigma_E)D(\varepsilon_\sigma). \]
	So it suffices to look at the morphism on $\gamma_2(L^\Q(\Gamma,2))$ induced by $U$. Fix any edge $e=\{v,w\}$ with $v<w$. Note that the only non-zero matrices in the column of $U$ corresponding with $[v]$ are those matrices in the row corresponding with $[v']$ where $[v']\olprec [v]$ (or equivalently $v'\prec v$). Hence, there are $a_{v'},b_{w'}\in \Z$ (for all $v'\prec v$ and $w'\prec w$) such that
    \[Uv=\sum_{\substack{v'\in V\\v'\prec v}}a_{v'}v' \quad \text{and} \quad Uw=\sum_{\substack{w'\in V\\w'\prec w}}b_{w'}w'.\]
    Thus we obtain that
    \begin{align*}
        Ue^{k(e)}&=U[v,w]=\left[\sum_{\substack{v'\in V\\v'\prec v}}a_{v'}v',\sum_{\substack{w'\in V\\w'\prec w}}b_{w'}w'\right]=\sum_{\substack{v'\in V\\v'\prec v}}\sum_{\substack{w'\in V\\w'\prec w}}a_{v'}b_{w'}[v',w']=\sum_{\substack{\{v',w'\}\in E\\\{v',w'\}\prec e}}a_{v'}b_{w'}[v',w']\\
        &\in \left\{\sum_{\substack{e'\in E\\e'\prec e}}a_{e'}(e')^{k(e')}\:\middle\vert\: a_{e'}\in \Z\right\}.
    \end{align*}
    Since $\mu_i\not\olprec \mu_j$ if $i>j$, it follows that the matrix (with respect to $(E,<)$) of the morphism on $\gamma_2(L^\Q(\Gamma,2))$ that is induced by $U$ is a block upper triangular matrix
    \[  \begin{pmatrix}
              B_1 & B_{12} & \hdots & B_{1s} \\
              0 & B_2 & \ddots & \vdots \\
              \vdots & \ddots & \ddots & B_{s-1 s} \\
              0 & \hdots & 0 & B_s
        \end{pmatrix} \]
    where $B_{ij} \in \Z^{|\mu_i| \times |\mu_j|}$ and $B_{ij}=0$ if $\mu_i\not\olprec\mu_j$. Hence, we obtain that the matrix of $\olvarphi\vert_{\gamma_2(L^\Q(\Gamma,2))}$ has the form
    \[ P(\sigma_E) \cdot D(\varepsilon_\sigma) \cdot
        \begin{pmatrix}
              B_1 & B_{12} & \hdots & B_{1s} \\
              0 & B_2 & \ddots & \vdots \\
              \vdots & \ddots & \ddots & B_{s-1 s} \\
              0 & \hdots & 0 & B_s
        \end{pmatrix}. \]
    Since this matrix coincides, by construction, with the matrix of $\varphi\vert_{\gamma_2(G)}\in \Aut(\gamma_2(G))$ it follows similarly as before that $B_i \in \Gl_{|\mu_i|}(\Z)$.

    By Theorem \ref{thm:GradedAutomorphismsLieAlgebra}, it holds that $p(\sigma)\in \Aut(\olGamma)$ is unique. The assignment $\varphi \to \overline{\varphi}$ from $\Aut(G)$ to $\Aut(L(\Gamma, 2))$ is a group morphism since the association of a Lie algebra to a group from section \ref{sec:LieAlgAssToGrps} is functorial. Using the last sentence of Theorem \ref{thm:GradedAutomorphismsLieAlgebra}, it follows that $\varphi \mapsto p(\sigma)$ defines a group homomorphism from $G$ to $\Aut(\overline{\Gamma})$.
\end{proof}


As stated at the end of the lemma above, any automorphism $\varphi\in \Aut(G_{\Gamma(k)})$ induces a unique permutation $p(\sigma) \in \Aut(\overline{\Gamma})$. In general, one does not obtain every permutation of $\Aut(\overline{\Gamma})$ in this way, as the weight function $k:E \to \N_0$ imposes conditions on $p(\sigma)$. Before describing these conditions, we introduce some necessary invariants of integer matrices and recall their relation with the Smith normal form.

\begin{definition}\label{DeterminantDivisors}
	Let $A\in \Z^{n\times n}$ be a matrix and $l=1,\dots,n$. Then the \textit{$l$-th determinant divisor} $d_l(A)$ of $A$ equals the greatest common divisor of the determinants of the $l\times l$ minors of $A$.\\
	In particular, if $A=\text{diag}(a_1,\dots,a_n)$ is a diagonal matrix, then $d_l(A)$ equals the greatest common divisor of all $l$-fold products of diagonal elements, i.e. for any $l=1,\dots, n$ it holds that
    \[ d_l(\text{diag}(a_1,\dots,a_n))= \gcd\left\{\prod_{i\in I}a_{i} \: \middle\vert
         I \subset \{1, \ldots, n\} \text{ with } |I| = l\right\}.\]
\end{definition}
The determinant divisors completely determine the Smith normal form of a matrix in $\Z$. This is summarized by the following lemma. Proofs of these facts and an introduction to the Smith normal form of a matrix can be found in most handbooks about basic algebra, for example in \cite[Part 1, section 1]{norm12}.

\begin{lemma}\label{lem:SNFMatrices}
	Let $A,B\in \Z^{n\times n}$ be two matrices, then the following are equivalent:
	\begin{enumerate}[label = (\roman*)]
		\item $A$ and $B$ have the same Smith normal form.
		\item there exist $P,Q\in \Gl_n(\Z)$ such that $A=PBQ$.
		\item $d_i(A)=d_i(B)$ for all $i=1,\dots,n$.
	\end{enumerate}
\end{lemma}


\begin{notation}
	Let $\Gamma=(V,E)$ be a graph with weight function $k:E\to \N_0$ on its edges and quotient graph $\olGamma=(\Lambda,\olE,\Psi)$. Recall from section \ref{sec:relationsOnVerticesAndEdges} that any edge $e = \{v, w\} \in E$ gives both an equivalence class $\mu = [e]$ and an edge in the quotient graph $\{[v], [w]\}$ which can be identified under the map from Equation (\ref{eq:identificationEbarAndM}). The edges in the equivalence class $\mu$ have weights $k\vert_{\mu}$ and thus the diagonal matrix with these weights on its diagonal equals $D(k\vert_{\mu})$. For any such $\mu \in M$ and $l=1,\dots,\vert\mu\vert$ we denote with $d_l(\mu)$ the $l$-th determinant divisor of the matrix $D(k\vert_{\mu})$, i.e.
    \[ d_l(\mu) := \gcd\left\{\prod_{e \in I} k(e) \: \middle| \:
         I \subseteq \mu \text{ with } |I| = l \right\}.\]
\end{notation}
Now we are ready to formulate a connection between the possible graph automorphisms $\sigma \in \Aut(\Gamma)$ that can occur in the matrix representation from Lemma \ref{lem:inducedAutoAbAndGamma2} and the weight function $k$ on the edges of $\Gamma$.
\begin{lemma}\label{lem:RestrictionsWeightsUsingDeterminantDivisors}
	Let $\Gamma=(V,E)$ be a finite undirected simple graph with weight function $k:E\to \N_0$ on its edges. Suppose that $\varphi\in \Aut(G_{\Gamma(k)})$ and let $\sigma \in \Aut(\Gamma)$ be a graph automorphism which can occur in the matrix representation of $\varphi$ given by Lemma \ref{lem:inducedAutoAbAndGamma2}. Let $e \in E$ be any edge, then for any $l = 1,\dots,\vert [e] \vert$ it holds that
	\[ d_l([e])=d_l([\sigma_E(e)]).\]
	In particular, it holds that:
	\[ \begin{cases}
		\underset{e'\in [e]}{\gcd}(k(e'))&=\underset{e'\in [\sigma_E(e)]}{\gcd}(k(e'))\\
        \noalign{\vskip9pt}
		\prod\limits_{e'\in [e]} k(e')&=\prod\limits_{e'\in [\sigma_E(e)]} k(e').
	\end{cases}\]
\end{lemma}
\begin{proof}
    By Lemma \ref{lem:inducedAutoAbAndGamma2}, we know that, with respect to the ordered basis $(E^k, <)$, the matrix representation of $\varphi$ restricted to $\gamma_2(G_{\Gamma(k)})$ is of the form
    \[P(\sigma_E) \cdot D(\varepsilon_\sigma) \cdot \underbrace{\begin{pmatrix}
		      B_1 & B_{12} & \hdots & B_{1s} \\
		      0 & B_2 & \ddots & \vdots \\
		      \vdots & \ddots & \ddots & B_{s-1 s} \\
		      0 & \hdots & 0 & B_s
	    \end{pmatrix}}_{=: B} \]
    where $B_i \in \Gl_{|\mu_i|}(\Z)$, $B_{ij} \in \Z^{|\mu_i| \times |\mu_j|}$ and $B_{ij}=0$ if $\mu_i\not\olprec\mu_j$. Now Lemma \ref{lem:AutomorphismOnSqrtGamma2} implies that
    \begin{equation}
        \label{eq:proofRestrictionGraphAutos}
        D(k) \cdot P(\sigma_{E}) \cdot D(\varepsilon_\sigma) \cdot B \cdot D(k)^{-1}\in \Gl_{\vert E\vert }(\Z).
    \end{equation}
    Remark that $D(k) \cdot P(\sigma_E)= P(\sigma_E) \cdot D( k\circ \sigma_E )$. Indeed, viewing these matrices as linear maps on $\Sp(E)$, it holds for any edge $e\in E$ that
	\begin{align*}
		D(k)P(\sigma_E)\, e &= D(k)\, \sigma_E(e)\\
        &= k(\sigma_E(e)) \, \sigma_E(e)\\
		&= k(\sigma_E(e)) P(\sigma_E) \, e\\
        &= P(\sigma_E) k(\sigma_E(e)) \, e\\
        &= P(\sigma_E) D(k \circ \sigma_E) \, e.
	\end{align*}
	Using this and the fact that $D(k \circ \sigma_E)$ and $D(\varepsilon_\sigma)$ commute, we find that
    \[ D(k) \cdot P(\sigma_{E}) \cdot D(\varepsilon_\sigma) \cdot B \cdot D(k)^{-1} = \Big( P(\sigma_E) \cdot D(\varepsilon_\sigma) \Big) \cdot \Big( D(k \circ \sigma_E) \cdot B \cdot D(k)^{-1} \Big),\]
    which, by Equation (\ref{eq:proofRestrictionGraphAutos}) must lie in $\Gl_{\vert E\vert }(\Z)$. Since $P(\sigma_E) \cdot D(\varepsilon_\sigma)$ is clearly a matrix in $\Gl_{\vert E\vert }(\Z)$, we find that 
    \[ C := D(k \circ \sigma_E) \cdot B \cdot D(k)^{-1} \in \Gl_{\vert E\vert }(\Z). \]
    Recall that $B$ is a block upper triangular matrix and thus $C$ is also block upper triangular. Note that if a block upper triangular matrix lies in $\Gl_n(\Z)$, then so does every one of its blocks on the diagonal. Thus by considering the block on the diagonal corresponding to $[e] = \mu = \mu_i$, we find that
    \[ D((k \circ \sigma_E)\vert_\mu) \cdot B_i \cdot D(k\vert_\mu) \in \Gl_{\vert \mu \vert}(\Z). \]
    Since $B_i$ also lies in $\Gl_{\vert\mu\vert}(\Z)$, Lemma \ref{lem:SNFMatrices} implies that $D((k \circ \sigma_E)\vert_\mu)$ and $D(k\vert_\mu)$ must have the same Smith normal form or, equivalently, that for any $l=1, \ldots, |\mu|$ we have
	\[ d_l(\sigma(\mu)) = d_l(D((k \circ \sigma_E)\vert_{\mu})) = d_l(D(k\vert_{\mu})) = d_l(\mu).\]
	This proves the main claim of the lemma. By noticing for all $\mu \in M$ that 
	\[ \begin{cases}
		d_1(\mu)&=\underset{e'\in \mu}{\gcd}(k(e'))\\
        \noalign{\vskip9pt}
		d_{\vert \mu \vert}(\mu)&=\prod\limits_{e'\in \mu} k(e'),
	\end{cases}\]
	we obtain the last claim of the statement.
\end{proof}

Recall that $M = E/\sim$ is the set of equivalence classes of edges and that for any $\sigma \in \Aut(\Gamma)$ we write $\sigma_E$ for the induced permutation on the edges. We define the following subgroup of $\Aut(\Gamma)$
\begin{equation}
    \label{eq:autGammak}
    \Aut(\Gamma(k)) := \Big\{ \sigma \in \Aut(\Gamma) \: \Big| \: \forall \mu \in M, 1 \leq l \leq |\mu|: d_l(\mu) = d_l(\sigma_E(\mu)) \Big\}.
\end{equation}
Then, for any field $K \subset \C$, define the following subgroup of $\Gl(\Sp_K(V))$
\[ \mathcal{G}_{\Gamma(k)} := P(\Aut(\Gamma(k))) \cdot \biggg(\prod_{\lambda\in \Lambda}\Gl(\Sp_K(\lambda))\biggg) \cdot \mathcal{U}. \]
Clearly $\mathcal{G}_{\Gamma(k)}$ is a subgroup of $\mathcal{G}_\Gamma$ as defined in Theorem \ref{thm:GradedAutomorphismsLieAlgebra}. Moreover, we can prove the following lemma about $\mathcal{G}_{\Gamma(k)}$, which will be useful in the following section.

\begin{lemma}
    \label{lem:linearAlgebraicGroupGGammak}
    Let $K$ be a subfield of $\C$. The group $\mathcal{G}_{\Gamma(k)}$ is a linear algebraic group that decomposes as a semi-direct product $\mathcal{L} \rtimes \mathcal{U}$ where
    \[ \mathcal{L} := P(\Aut(\Gamma(k))) \cdot \biggg(\prod_{\lambda\in \Lambda}\Gl(\Sp_K(\lambda))\biggg)\]
    is a linearly reductive group and $\mathcal{U}$ is the unipotent radical of $\mathcal{G}_{\Gamma(k)}$.
\end{lemma}
\begin{proof}
    To prove that $\mathcal{G}_{\Gamma(k)}$ is a linear algebraic group, it suffices to show that it is Zariski closed as a subset of $\Gl(\Sp(V))$ (with coordinates with respect to the basis of vertices $V$). This can be easily seen to be the case since it can be written as a finite union of Zariski closed subsets:
    \[ \bigcup_{\sigma \in \Aut(\Gamma(k))} P(\sigma) \cdot \biggg(\prod_{\lambda\in \Lambda}\Gl(\Sp_K(\lambda))\biggg) \cdot \mathcal{U}. \]
    To see that $\mathcal{L}$ is linearly reductive, it suffices to show that its Zariski-connected component at the identity is linearly reductive. This Zariski connected component is equal to $\prod_{\lambda\in \Lambda}\Gl(\Sp_K(\lambda))$. It must be linearly reductive as it is a direct product of general linear groups, which are known to be linearly reductive. That $\mathcal{U}$ is the unipotent radical of $\mathcal{G}_{\Gamma(k)}$ follows from the fact the $\mathcal{U}$ is the unipotent radical of $\mathcal{G}$ (see Theorem \ref{thm:GradedAutomorphismsLieAlgebra}). Indeed, $\mathcal{G}_{\Gamma(k)}$ and $\mathcal{G}$ have the same Zariski-connected component at the identity and thus the same unipotent radical.
\end{proof}

At last, from Lemma \ref{lem:RestrictionsWeightsUsingDeterminantDivisors}, we can now conclude the following corollary. The field in this corollary is $K = \Q$. Recall that $\pi:\Aut_g(L^\Q(\Gamma, 2)) \to \Gl(\Sp_\Q(V))$ is the restriction map. 

\begin{corollary}
\label{cor:inducedAutoLiesInLAGweights}
    Let $\Gamma = (V, E)$ be a finite undirected simple graph with weight function $k:E \to \N_0$ on its edges. For any automorphism $\varphi \in \Aut(G_{\Gamma(k)})$ it holds that $\pi(\overline{\varphi})$ lies in $\mathcal{G}_{\Gamma(k)}$, where $\overline{\varphi}$ denotes the induced automorphism on $L^\Q(\Gamma, 2)$.
\end{corollary}

\newpage
\section{The \texorpdfstring{$R_{\infty}$}{Roo}--property for the groups \texorpdfstring{$G_{\Gamma(k)}$}{G_Gamma(k)}}

For any group $G$ we define the \textit{$R_{\infty}$--nilpotency index} as the least integer $c$ such that $G/\gamma_{c+1}(G)$ has the $R_{\infty}$--property. If such an integer does not exist, we define it to be infinite. Recall that by construction $G_{\Gamma}=A(\Gamma)/\gamma_3(A(\Gamma))$. In \cite{witd23} the author gives a lower and upper bound for the $R_{\infty}$--nilpotency index of $A(\Gamma)$.

\begin{theorem}{\cite[Theorem 1.6]{witd23}}\label{thm:LowerUpperBoundNilpIndex}
	If $\Gamma$ is a non-empty finite undirected simple graph, then $A(\Gamma)$ has the $R_{\infty}$--property and its $R_{\infty}$--nilpotency index $c$ satisfies
	\[ \xi(\Gamma)\leq c\leq \Xi(c) \]
	where
	\[ \xi(\Gamma):=\min_{\{\lambda_1,\lambda_2\}\in \olE} (\vert \lambda_1\vert +\vert \lambda_2 \vert)\quad \text{and}\]
	\[ \Xi(\Gamma):=\min_{\{\lambda_1,\lambda_2\}\in \olE} c(\lambda_1,\lambda_2)\quad \text{with}\quad 
	c(\lambda_1,\lambda_2):=\begin{cases}
		\max\left\{2\vert \lambda_1\vert +\vert \lambda_2 \vert, \vert \lambda_1 \vert +2\vert \lambda_2 \vert\right\}&\text{if }\lambda_1\neq \lambda_2\\
		2\vert \lambda_1\vert &\text{if }\lambda_1=\lambda_2
	\end{cases}.\]
\end{theorem}

\begin{theorem}\label{thm:MainResult}
	Let $\Gamma=(V,E)$ be a finite undirected simple graph and denote with $\Gamma_0=(V_0,E_0)$ the subgraph of $\Gamma$ containing the vertices belonging to a coherent component of size one. Then the following holds.
	\begin{enumerate}[label = (\roman*)]
		\item If $E_0=\emptyset$, then $G_{\Gamma}$ does not have the $R_{\infty}$--property.
		\item If $\Gamma_0$ has edges and $m\in \N \setminus \{0, 1\}$, then there exists a weight function $k:E\to \N_0$ such that $G_{\Gamma(k)}$ has the $R_{\infty}$--property and $\vert G_{\Gamma(k)}:G_{\Gamma}\vert=m$.
		\item If $\Gamma_0$ has edges and $\Gamma = \Gamma_0$, then $G_{\Gamma}$ has the $R_{\infty}$--property. In particular, all finitely generated torsion-free 2-step nilpotent groups commensurable with $G_{\Gamma}$ have the $R_{\infty}$--property.
	\end{enumerate}
\end{theorem}
\begin{proof}
	\begin{enumerate}[label = (\roman*)]
		\item If $V_0=\emptyset$, then all coherent components contain at least two vertices. Hence, it follows that $\xi(\Gamma)\geq 4$. By Theorem \ref{thm:LowerUpperBoundNilpIndex} we get that the $R_{\infty}$--nilpotency index of $A(\Gamma)$ is at least $4$. Thus $G_{\Gamma}=A(\Gamma)/\gamma_2(A(\Gamma))$ does not have the $R_{\infty}$--property.
    
        Let us assume that $V_0\neq \emptyset$. By Lemma 6.2 in \cite{witd23} it follows that there exist monic polynomials $p_\lambda(X)\in \Z[X]$ (for $\lambda\in \Lambda\setminus (V_0/\sim)$) of degree $|\lambda|\geq 2$ such that, if $\alpha_{\lambda 1},\dots,\alpha_{\lambda |\lambda|}\in \C$ are the zeros of $p_\lambda(X)$, it holds for all $\lambda,\lambda'\in \Lambda\setminus (V_0/\sim)$ (possibly equal), $i=1,\dots,|\lambda|$ and $j=1,\dots,|\lambda'|$ that
        \begin{equation}\label{eq:eigenvalueConstruction}
            \alpha_{\lambda i}\neq \pm 1,\: \alpha_{\lambda i}\alpha_{\lambda' j}\neq 1 \quad \text{and} \quad \alpha_{\lambda 1}\dots\alpha_{\lambda|\lambda|}=-1.
        \end{equation}
        
        Denote for any $\lambda\in \Lambda$ with $A_\lambda\in \Gl_{|\lambda|}(\Z)\subset \Gl(\Sp_\Q(\lambda))$ the linear map for which the matrix with respect to the $\Z$-basis $(\lambda,<)$ is the companion matrix of $p_\lambda(X)$ if $\lambda\not\in V_0/\sim$ or is $\begin{pmatrix}-1\end{pmatrix}$ if $\lambda \in V_0/\sim$. Combining these linear maps, leads to the linear map 
        \[A:=\prod_{\lambda\in \Lambda} A_\lambda\in \Gl_{|V|}(\Z)\subset \prod_{\lambda\in \Lambda}\Gl(\Sp_\Q(\lambda)).\]
        Recall that we defined a projection map
        \begin{equation}\label{eq:projectionPi}
            \pi:\Aut_g(L^\Q(\Gamma, 2)) \to \mathcal{G}_\Gamma: f \mapsto f|_{\Sp_\Q(V)}.
        \end{equation}
        Lemma 3.8 in \cite{witd23} argues that there exists an automorphism $\varphi\in \Aut(G_\Gamma)$ such that the induced automorphism $\olvarphi \in \Aut_g(L^\Q(\Gamma, 2))$ projects onto $\pi(\olvarphi)=A$. By construction, we know that $\olvarphi\vert_{\Sp_\Q(V)}=\olvarphi\vert_{\Sp_\Q(V_0)}\times\olvarphi\vert_{\Sp_\Q(V\setminus V_0)}$ is diagonalizable. Hence, we can fix a basis of eigenvectors $\mathcal{V}_1\subset \Sp_\Q(V_0)$ of $\olvarphi\vert_{\Sp_\Q(V_0)}$ and $\mathcal{V}_2\subset \Sp_\Q(V\setminus V_0)$ of $\olvarphi\vert_{\Sp_\Q(V\setminus V_0)}$. Thus $\mathcal{V}:=\mathcal{V}_1\cup \mathcal{V}_2$ is a basis of eigenvectors of $\olvarphi\vert_{\Sp_\Q(V)}$ for $\Sp_\Q(V)$. Note that $[\mathcal{V},\mathcal{V}]$ is a generating set of $\gamma_2(L^\Q(\Gamma, 2))$. Since $[\Sp_\Q(V_0),\Sp_\Q(V_0)]=\Sp_\Q(E_0)$ is trivial, it follows that
        \[[\mathcal{V},\mathcal{V}]=[\mathcal{V}_1,\mathcal{V}_2]\cup [\mathcal{V}_2,\mathcal{V}_2].\]
        Note that the non-zero brackets in $[\mathcal{V}_1,\mathcal{V}_2]$ and $[\mathcal{V}_2,\mathcal{V}_2]$ are eigenvectors of $\olvarphi$. Moreover, the corresponding eigenvalues of the latter are products of two (possibly equal) $\alpha_{\lambda i}$ in Equation (\ref{eq:eigenvalueConstruction}) and thus are not equal to 1. Since $\olvarphi\vert_{\Sp_\Q(V_0)}=-\Id$, the eigenvalues corresponding with the eigenvectors in $[\mathcal{V}_1,\mathcal{V}_2]$ are of the form $-\alpha_{\lambda i}$ and are thus by Equation (\ref{eq:eigenvalueConstruction}) also not equal to 1. The non-zero brackets in $[\mathcal{V},\mathcal{V}]$ form a basis of eigenvectors of $\olvarphi$ for $\Sp_\Q(E)=\gamma_2(L^\Q(\Gamma, 2))$. Hence, $\mathcal{V}\cup [\mathcal{V},\mathcal{V}]$ forms a basis of $L^\Q(\Gamma, 2)$ of eigenvectors of $\olvarphi$ with all eigenvalues not equal to 1. By Lemma \ref{lem:eigenvalue1} it now follows that $R(\varphi)<\infty$ which concludes the proof.
        
		\item Fix any $m\in \N \setminus \{0, 1\}$ and any edge $e_0 = \{v_0,w_0\} \in E_0$ (with $v_0,w_0\in V_0$). Note that the equivalence class $[e_0] =\{\{v_0,w_0\}\}\subset E$ is a singleton, by definition. Define $k:E\to \N_0$ by setting
		\[ k:E\to \N_0:e\mapsto \begin{cases}
			m &\text{if } e=\{v_0,w_0\}\\
			1 &\text{if } e\neq \{v_0,w_0\}
		\end{cases}.\]
		By Lemma \ref{lem:structure GGammak} it follows that $\vert G_{\Gamma(k)}:G_{\Gamma}\vert=\prod_{e\in E}k(e)=m$. We prove that $G_{\Gamma(k)}$ has the $R_{\infty}$--property.

        First, recall the definition of $\Aut(\Gamma(k))$ in Equation (\ref{eq:autGammak}) and note that with the above choice of $k$, for any $\sigma \in \Aut(\Gamma(k))$, it holds that $\sigma_E(e_0) = e_0$. Next, take any automorphism $\varphi \in \Aut(G_{\Gamma(k)})$. Recall the definition of $\pi$ from Equation (\ref{eq:projectionPi}). Since $\pi$ is in fact bijective, it makes sense to write $\pi^{-1}(g)$ for an element $g \in \mathcal{G}_\Gamma$.
        
        By Corollary \ref{cor:inducedAutoLiesInLAGweights}, we know that $\pi(\overline{\varphi})$ lies in the linear algebraic group $\mathcal{G}_{\Gamma(k)} \subset \mathcal{G}_\Gamma$ (over $\Q$). Thus, the same must hold for its semi-simple part $\pi(\overline{\varphi})_s$. By \cite[Theorem 4.3, chapter VIII]{hoch81-1} and Lemma \ref{lem:linearAlgebraicGroupGGammak}, there must exist an $h \in \mathcal{G}_{\Gamma(k)}$ such that $g := h \pi(\overline{\varphi})_s h^{-1}$ lies in $\mathcal{L} \subset \mathcal{G}_{\Gamma(k)}$. Note that $g$ is thus represented by a matrix of the form
        \[ P(\sigma) \begin{pmatrix}
		      A_1 &  0& \hdots & 0 \\
		      0 & A_2 & \ddots & \vdots \\
		      \vdots & \ddots & \ddots & 0 \\
		      0 & \hdots & 0 & A_r
	    \end{pmatrix}, \]
        where $\sigma \in \Aut(\Gamma(k))$ (and thus $\sigma_E(e_0) = e_0$) and $A_i \in \Gl_{|\lambda_i|}(\Q)$ for all $i =1, \ldots, r$. Note that from the fact that $\sigma_E(e_0) = e_0$ and $[v_0], [w_0]$ are singletons, the subspace $\Sp_\Q\{v_0, w_0\}$ is preserved by $g$. Now, since $\pi(\overline{\varphi})$ is integer-like, so is $g$ and since $g$ is given by a matrix over $\Q$ of the above form, it follows that each $A_i$ is integer-like. In particular, if $A_i$ has dimension $1 \times 1$, it must be equal to $(\pm 1)$. This implies that if we restrict $g$ to $\Sp_\Q\{v_0, w_0\}$, we get with respect to $v_0, w_0$, the matrix representation 
        \[ \begin{pmatrix}
            a & 0\\
            0 & b
        \end{pmatrix} \quad \text{or} \quad \begin{pmatrix}
            0 & 1\\
            1 & 0
        \end{pmatrix}\begin{pmatrix}
            a & 0\\
            0 & b
        \end{pmatrix},\]
        with $a,b \in \{-1, 1\}$. Let us prove that in either case, the automorphism $\pi^{-1}(g) \in \Aut_g(L^\Q(\Gamma, 2))$ has an eigenvalue 1.
        
        In the first case, if either $a$ or $b$ is equal to $1$, clearly $g$ has an eigenvalue 1 and thus also $\pi^{-1}(g)$ has. If both $a$ and $b$ are equal to $-1$, then the Lie bracket $[v_0, w_0]$ will be an eigenvector of $\pi^{-1}(g)$ with eigenvalue 1. Note that $[v_0, w_0]$ is non-zero since $\{v_0, w_0\} \in E$.

        In the second case, if $a = b = 1$, then $v_0+w_0$ will be an eigenvector of $g$ with eigenvalue 1. If $a = b = -1$, then $v_0 - w_0$ will be an eigenvector of $g$ with eigenvalue 1. In both cases we thus also get that $\pi^{-1}(g)$ has an eigenvalue 1. If $a$ and $b$ do not have the same sign, then the (non-zero) Lie bracket $[v_0, w_0]$ will be an eigenvector of $\pi^{-1}(g)$ with eigenvalue 1.

        We can thus conclude that $\pi^{-1}(g)$ always has an eigenvalue 1 and thus since they have the same eigenvalues, the same holds for $\overline{\varphi} \in L^\Q(\Gamma, 2)$. We can conclude by Lemma \ref{lem:eigenvalue1} that $R(\varphi) = \infty$ and thus, since $\varphi$ was an arbitrary automorphism of $G_{\Gamma(k)}$, that $G_{\Gamma(k)}$ has the $R_\infty$--property.
        
		\item If $\Gamma_0=\Gamma$, then all coherent components are singletons (or equivalently $\Gamma$ is transposition-free). In \cite[Theorem 1.4]{witd23} it is proven that $A(\Gamma)$ has $R_{\infty}$--nilpotency index 2. Thus $G_{\Gamma}$ has the $R_{\infty}$--property. By Proposition \ref{prop:isoGradLieAlgImplicationRinfty} the result follows.
	\end{enumerate}
\end{proof}

We give some examples where we can say more than the statement in Theorem \ref{thm:MainResult}.
\begin{example}[$\Gamma_0\subsetneq \Gamma$ has edges and $G_{\Gamma}$ has the $R_{\infty}$--property]
	Let us consider the next graph $\Gamma$. We also show its quotient graph $\olGamma$.\\
	\hbox to \hsize{\hfil{
			\begin{tikzpicture}		
				\node[regular polygon, regular polygon sides=4, minimum height=1.5cm] (g1) at (0,0) {};
				\foreach \n in {1,...,4} {\node[outer sep=0, inner sep=0] (g1\n) at (g1.corner \n) {};
					\draw[fill] (g1\n) circle (2.25pt);};
				\draw[edge] (g13)--(g12)--(g11)--(g14)--(g12);
				\node[anchor=south east,xshift=-5pt,yshift=5pt] at (g12.north west) {$\Gamma$};
				
				\def\grootte{{"2","1","1"}};
				\node[regular polygon, regular polygon sides=4, minimum height=1.5cm] (g2) at (3,0) {};
				\node[regular polygon, regular polygon sides=4, minimum height=2cm] (l2) at (3,0) {};
				\foreach \n in {1,...,4} {\node[outer sep=0, inner sep=0] (g2\n) at (g2.corner \n) {};};
				\foreach \n in {1,...,3} {\draw[fill] (g2\n) circle (2.25pt);
					\node[outer sep=0, inner sep=0] (l2\n) at (l2.corner \n) {\strut\pgfmathprint{\grootte[\n-1]}};};
				\draw[edge] (g21) to [out=230,in=310,looseness=50] (g21)--(g22)--(g23);
				\node[anchor=south east,xshift=-5pt,yshift=5pt] at (g22.north west) {$\olGamma$};
			\end{tikzpicture}
		}\hfil}
	In \cite[Theorem 4.4 and Example 4.5]{dl23-1}, we proved that $G_{\Gamma}$ has the $R_{\infty}$--property (where we used the complement graph to represent $G_{\Gamma}$). Proposition \ref{prop:isoGradLieAlgImplicationRinfty} now tells us that all groups commensurable with $G_{\Gamma}$ have the $R_{\infty}$--property. In particular, for any weight function $k:E\to \N_0$ on the edges of $\Gamma$ it holds that $G_{\Gamma(k)}$ has the $R_{\infty}$--property.
			%
			%
\end{example}

\begin{example}[$G_{\Gamma}$ does not have the $R_{\infty}$--property, but $G_{\Gamma(k_n)}$ does]

    \label{ex:mainCounterexample}    
	We could for example consider the next graph with weight function $k_n:E\to \N_0$ (for any $n\in \N\setminus \{0,1\}$) on its edges. We fixed a total order of $V$, $E$ and $\Lambda=V/\sim$ as introduced in section \ref{sec:relationsOnVerticesAndEdges}.
 
	\begin{figure}[H]
		\begin{subfigure}{0.5\textwidth}
			\centering
                \captionsetup{justification=centering}
			\begin{tikzpicture}[main_node/.style={circle,draw,fill,minimum size=4.5pt,inner sep=0, outer sep=0]}]
				\def\lengte{1.5cm}
				\node[main_node] (0) at (0,{2*\lengte}) {};
				\node[anchor=south west,xshift=2pt,yshift=2pt,outer sep=0, inner sep=0] at (0.north east) {\scalebox{0.8}{$v_1$}};
				\node[main_node] (1) at (0,\lengte) {};
				\node[anchor=north west,xshift=2pt,yshift=-2pt,outer sep=0, inner sep=0] at (1.south east) {\scalebox{0.8}{$v_7$}};
				\node[main_node] (2) at (0,0) {};
				\node[anchor=north west,xshift=2pt,yshift=-2pt,outer sep=0, inner sep=0] at (2.south east) {\scalebox{0.8}{$v_2$}};
				\node[main_node] (3) at (\lengte,{2*\lengte}) {};
				\node[anchor=south west,xshift=2pt,yshift=2pt,outer sep=0, inner sep=0] at (3.north east) {\scalebox{0.8}{$v_3$}};
				\node[main_node] (4) at (\lengte,\lengte) {};
				\node[anchor=north west,xshift=2pt,yshift=-2pt,outer sep=0, inner sep=0] at (4.south east) {\scalebox{0.8}{$v_8$}};
				\node[main_node] (5) at (\lengte,0) {};
				\node[anchor=north west,xshift=2pt,yshift=-2pt,outer sep=0, inner sep=0] at (5.south east) {\scalebox{0.8}{$v_4$}};
				\node[main_node] (6) at ({2*\lengte},{2*\lengte}) {};
				\node[anchor=south west,xshift=2pt,yshift=2pt,outer sep=0, inner sep=0] at (6.north east) {\scalebox{0.8}{$v_5$}};
				\node[main_node] (7) at ({2*\lengte},\lengte) {};
				\node[anchor=north west,xshift=2pt,yshift=-2pt,outer sep=0, inner sep=0] at (7.south east) {\scalebox{0.8}{$v_9$}};
				\node[main_node] (8) at ({2*\lengte},0) {};
				\node[anchor=north west,xshift=2pt,yshift=-2pt,outer sep=0, inner sep=0] at (8.south east) {\scalebox{0.8}{$v_6$}};
				
				\path[draw, line width=0.75pt]
				(0) edge 	node[right] {\footnotesize{1}} (1) 
				(1) edge 	node[right] {\footnotesize{1}} (2)
				(1) edge 	node[above] {\footnotesize{$n$}} (4) 
				(3) edge 	node[right] {\footnotesize{1}} (4) 
				(4) edge 	node[right] {\footnotesize{1}} (5) 
				(4) edge	node[above] {\footnotesize{1}} (7) 
				(6) edge	node[right] {\footnotesize{1}} (7) 
				(7) edge	node[right] {\footnotesize{1}} (8)

                (0) edge 	node[left] {\scalebox{0.8}{$e_1$}} (1) 
				(1) edge 	node[left] {\scalebox{0.8}{$e_2$}} (2)
				(1) edge 	node[below] {\scalebox{0.8}{$e_7$}} (4) 
				(3) edge 	node[left] {\scalebox{0.8}{$e_3$}} (4) 
				(4) edge 	node[left] {\scalebox{0.8}{$e_4$}} (5) 
				(4) edge	node[below] {\scalebox{0.8}{$e_8$}} (7) 
				(6) edge	node[left] {\scalebox{0.8}{$e_5$}} (7) 
				(7) edge	node[left] {\scalebox{0.8}{$e_6$}} (8)
				;
			\end{tikzpicture}
            \caption{Graph $\Gamma=(V,E)$ with weights on its edges and with total orders for $V$ and $E$}
		\end{subfigure}
		\begin{subfigure}{0.5\textwidth}
			\centering
                \captionsetup{justification=centering}
			\begin{tikzpicture}[main_node/.style={circle,draw,fill,minimum size=4.5pt,inner sep=0, outer sep=0]}]
				\def\lengte{1.5cm}
				\node[main_node] (0) at (0,\lengte) {};
				\node[anchor=south east,xshift=-2pt,yshift=2pt,outer sep=0, inner sep=0] at (0.north west) {\footnotesize{2}};
				\node[anchor=north west,xshift=2pt,yshift=-2pt,outer sep=0, inner sep=0] at (0.south east) {\scalebox{0.8}{$\lambda_1$}};
				\node[main_node] (1) at (0,0) {};
				\node[anchor=north east,xshift=-2pt,yshift=-2pt,outer sep=0, inner sep=0] at (1.south west) {\footnotesize{1}};
				\node[anchor=south west,xshift=2pt,yshift=2pt,outer sep=0, inner sep=0] at (1.north east) {\scalebox{0.8}{$\lambda_4$}};
				\node[main_node] (2) at (\lengte,\lengte) {};
				\node[anchor=south east,xshift=-2pt,yshift=2pt,outer sep=0, inner sep=0] at (2.north west) {\footnotesize{2}};
				\node[anchor=north west,xshift=2pt,yshift=-2pt,outer sep=0, inner sep=0] at (2.south east) {\scalebox{0.8}{$\lambda_2$}};
				\node[main_node] (3) at (\lengte,0) {};
				\node[anchor=north east,xshift=-2pt,yshift=-2pt,outer sep=0, inner sep=0] at (3.south west) {\footnotesize{1}};
				\node[anchor=south west,xshift=2pt,yshift=2pt,outer sep=0, inner sep=0] at (3.north east) {\scalebox{0.8}{$\lambda_5$}};
				\node[main_node] (4) at ({2*\lengte},\lengte) {};
				\node[anchor=south east,xshift=-2pt,yshift=2pt,outer sep=0, inner sep=0] at (4.north west) {\footnotesize{2}};
				\node[anchor=north west,xshift=2pt,yshift=-2pt,outer sep=0, inner sep=0] at (4.south east) {\scalebox{0.8}{$\lambda_3$}};
				\node[main_node] (5) at ({2*\lengte},0) {};
				\node[anchor=north east,xshift=-2pt,yshift=-2pt,outer sep=0, inner sep=0] at (5.south west) {\footnotesize{1}};
				\node[anchor=south west,xshift=2pt,yshift=2pt,outer sep=0, inner sep=0] at (5.north east) {\scalebox{0.8}{$\lambda_6$}};
				
				\path[draw, line width=0.75pt]
				(0) edge node {} (1) 
				(1) edge node {} (3) 
				(3) edge node {} (2)
				(3) edge node {} (5) 
				(4) edge node {} (5) 
				;
			\end{tikzpicture}
                \caption{Quotient graph $\olGamma=(\Lambda,\olE,\Psi)$\\ with total order on $\Lambda$}
		\end{subfigure}
	\end{figure}
        For any $n\in \N\setminus \{0,1\}$ we consider the group $G_{\Gamma(k_n)}$ as defined in section \ref{sec:definition Ggamma(k)}. By the proof of Theorem \ref{thm:MainResult} (ii) we already know that $G_{\Gamma(k_n)}$ has the $R_{\infty}$--property for any $n\in \N\setminus \{0,1\}$.\\
        We now prove that $G_{\Gamma}$ does not have the $R_{\infty}$--property by giving an explicit automorphism $\varphi\in \Aut(G_{\Gamma})$ with $R(\varphi)<\infty$. Denote with $\sigma:=(v_1\: v_5)(v_7 \: v_9)(v_2\: v_6)\in \Aut(\Gamma)$. Define the matrix $A\in \Gl_2(\Z)$ and $B\in \Gl_9(\Z)$ by
	\[ A:=\begin{pmatrix}
		2 & 1 \\
		1 & 1
	\end{pmatrix} \quad \text{and}\quad
	B:=\begin{pmatrix}
		0 & 0 & A & 0 & 0 & 0 \\
		0 & A & 0 & 0 & 0 & 0 \\
		\mathds{1}_2 & 0 & 0 & 0 & 0 & 0 \\
		0 & 0 & 0 & 0 & 0 & -1 \\
		0 & 0 & 0 & 0 & -1 & 0 \\
		0 & 0 & 0 & 1 & 0 & 0
	\end{pmatrix}=P(\sigma)\cdot \begin{pmatrix}
		\mathds{1}_2 & 0 & 0 & 0 & 0 & 0 \\
		0 & A & 0 & 0 & 0 & 0 \\
		0 & 0 & A & 0 & 0 & 0 \\
		0 & 0 & 0 & 1 & 0 & 0 \\
		0 & 0 & 0 & 0 & -1 & 0 \\
		0 & 0 & 0 & 0 & 0 & -1
	\end{pmatrix}\]
	where the second equality for $B$ will illustrate Lemma \ref{lem:inducedAutoAbAndGamma2}. Define the automorphism $\varphi\in \Aut(G_{\Gamma})$ by defining it on the $\Z$-basis $(V/\gamma_2(G_{\Gamma}),<)$ of $G_{\Gamma}/\gamma_2(G_{\Gamma})$ by the matrix $B\in \Gl_9(\Z)$. One can derive that the matrix of $\varphi\vert_{\gamma_2(G_\Gamma)}$ with respect to the $\Z$-basis $(E,<)$ equals
    \[C:=\begin{pmatrix}
		0 & 0 & -A & 0 & 0 \\
		0 & -A & 0 & 0 & 0 \\
		\mathds{1}_2 & 0 & 0 & 0 & 0 \\
		0 & 0 & 0 & 0 & -1 \\
		0 & 0 & 0 & 1 & 0 
	\end{pmatrix}=P(\sigma_E)D(\varepsilon_\sigma)\begin{pmatrix}
		\mathds{1}_2 & 0 & 0 & 0 & 0 \\
		0 & -A & 0 & 0 & 0 \\
		0 & 0 & -A & 0 & 0 \\
		0 & 0 & 0 & -1 & 0 \\
		0 & 0 & 0 & 0 & 1 
	\end{pmatrix}\]
    where $\sigma_E:=(e_1\: e_5)(e_7\: e_8)(e_2\: e_6)$ and $\varepsilon_\sigma(e_i)=\begin{cases}
        -1 &\text{if } i=7,8\\
        1 &\text{else}
    \end{cases}$. In particular, the eigenvalues of the induced automorphism $\olvarphi\in \Aut(L^\Q(G_\Gamma))$ are precisely the eigenvalues of the matrices $B$ and $C$. Denote with $\alpha_1:=\frac{ 3-\sqrt{5}}{2}$ and $\alpha_2:=\frac{3+\sqrt{5}}{2}$ the eigenvalues of $A$. One can calculate that $B$ has eigenvalues
    \[\alpha_1,\alpha_2,\pm\sqrt{\alpha_1},\pm \sqrt{\alpha_2},-1 \text{ and } \pm i\]
    and the eigenvalues of $C$ are
    \[-\alpha_1,-\alpha_2,\pm\sqrt{\alpha_1}i,\pm \sqrt{\alpha_2}i \text{ and } \pm i.\]
    Since 1 is not an eigenvalue of $\olvarphi$, Lemma \ref{lem:eigenvalue1} implies that $R(\varphi)<\infty$ and thus $G_\Gamma$ does not have the $R_\infty$--property.
\end{example}

\begin{corollary}
	There exist finitely generated torsion-free 2-step nilpotent groups which do not have the $R_\infty$--property but admit for any $n \in \N \setminus \{0, 1\}$ a subgroup of index $n$ that does have the $R_\infty$--property.
\end{corollary}

\newpage
\bibliographystyle{alpha}
\bibliography{references}
\end{document}